\documentclass[12pt,reqno]{amsart}
 \usepackage[fontsize=14pt]{fontsize}
 
\usepackage[usenames,dvipsnames]{color}
\usepackage{amsmath,amssymb,amsthm,graphicx,mathrsfs,url,latexsym,enumerate}
\usepackage{color}
\usepackage{a4wide}
\usepackage{ifthen}
\usepackage{verbatim}
\usepackage{fancyhdr}
\usepackage{url}
\usepackage{bbm}
\usepackage[sans]{dsfont}
\usepackage{MnSymbol}

\usepackage{amsbsy}
\usepackage[colorlinks=true,linkcolor=red,citecolor=green]{hyperref}
\usepackage{wrapfig}
\usepackage{tikz}
\usepackage{geometry}

\definecolor{bleu}{RGB}{27,88,145}
\definecolor{mauve}{RGB}{138,20,79}

\renewcommand{\Re}{\operatorname{Re}}

\newcommand{\Id}{\operatorname{Id}}

\newcommand{\supp}{\operatorname{supp}}

\newcommand{\Span}{\operatorname{Span}}

\newcommand{\N}{\mathbf N}
\newcommand{\R}{\mathbf R}

\newcommand{\Hess}{\operatorname{Hess}}
\newcommand{\Jac}{\operatorname{Jac}}

\newcommand{\Ran}{\operatorname{Ran}}

\newcommand{\argmin}{\operatorname{argmin}}

\newcommand{\tr}{\operatorname{tr}}
\renewcommand{\div}{\operatorname{div}}

\def\<{\langle}
\def\>{\rangle}
\newcommand{\bp}{{\it Proof. }}
\newcommand{\ep}{\hfill $\square$\\}

\def \jac {{\rm \, Jac }}

\newcommand{\be}{\begin{equation}}
\newcommand{\ee}{\end{equation}}
\newcommand{\bes}{\begin{equation*}}
\newcommand{\ees}{\end{equation*}}

\newcommand {\pa}{\partial}

\def\Re {{\rm \,Re\;}}

\def \tr{{\rm \, Tr\;}}

\def \jac {{\rm \, Jac }}

\numberwithin{equation}{section}
\numberwithin{figure}{section}

\newtheorem{theorem}{Theorem}
\newtheorem*{theorem*}{Theorem}
\newtheorem{proposition}{Proposition}

\newtheorem*{definition*}{Definition}
\newtheorem{lemma}[proposition]{Lemma}

\newtheorem{remark}[proposition]{Remark}

\newtheorem*{assumption*}{Assumption}

 \title[Exit time  of non-Gibbsian elliptic processes]{Eyring-Kramers formula for the mean exit time of  non-Gibbsian elliptic processes: the non characteristic boundary case}
 
\author[D. Le Peutrec]{Dorian Le Peutrec}
\address{D. Le Peutrec, Laboratoire de Mathématiques Jean Leray, 
Université de Nantes}
\email{dorian.lepeutrec@univ-nantes.fr}

\author[L. Michel]{Laurent Michel}
\address{L. Michel, Institut Math\'ematique de Bordeaux, Universit\'e de Bordeaux}
\email{laurent.michel@math.u-bordeaux.fr}

\author[B. Nectoux]{Boris Nectoux}
\address{B. Nectoux, Laboratoire de Mathématiques Blaise Pascal, UCA}
\email{boris.nectoux@uca.fr}

\begin{document} 
 \maketitle
    \begin{abstract}
 In this  work, we  derive  a new sharp asymptotic equivalent in the small temperature regime   $h\to 0$ for the mean exit time  from a bounded  domain      for the non-reversible process
  $dX_t=b(X_t)dt + \sqrt h \, dB_t$ under a generic orthogonal decomposition  of $b$ and when the boundary of $\Omega$ is assumed to be \textit{non characteristic}.    
The main contribution  of   this work   lies in the fact that we do not assume  that the process $(X_t,t\ge 0)$  is \textit{Gibbsian}. In this case, a new correction term  characterizing the \textit{non-Gibbsianness} of the process appears in the   equivalent of the mean exit time. The proof is mainly based on tools from spectral  and semi-classical analysis. 
    \medskip

   
 \end{abstract} 

\section{Introduction and main result}\label{sec1}


\subsection{Setting and assumptions}\label{sec.geo}
Let  $(\mathbf \Omega, \mathcal F, (\mathcal F_t)_{t\ge 0}, \mathbf P)$ be   a  filtered probability space, where the filtration satisfies the usual conditions, and   $(B_t,t\ge 0)$ be a $\mathbf R^d$-standard Brownian motion. 
Let $(X_t,\,t\ge 0)$ be the process    solution to the following elliptic It\^o stochastic differential equation on $\mathbf R^d$,
\begin{equation}\label{eq.langevin}
dX_t= b(X_t)\, dt +\sqrt h \, dB_t,
\end{equation} 
where $h>0$ and $b:\mathbf R^d\to \mathbf R^d$  is a smooth vector field.  The parameter $h>0$ is the temperature of the system. In this work, we   consider    the small temperature regime, i.e. $h\ll 1$, see Section \ref{sec.purp}, and we    make the following assumptions on the vector field $b$. There exist $\ell :\mathbf R^d\to \mathbf R^d$ and  $f:\mathbf R^d\to \mathbf R$  which are both $\mathcal C^\infty$ such that:
\begin{enumerate}
\item [] \textbf{[A$\perp$]}  The vector field  $b$ satisfies the following orthogonal decomposition:
$$
\text{for all } x\in \mathbf R^d, \  b(x)=-(\nabla f(x)  +  \ell (x) )
 \quad\text{and}\quad \ell (x) \cdot \nabla f(x) =0.$$

\item []\textbf{[A$_{x_0}$]} The function $f$ has a unique critical point $x_0$ in  $\mathbf R^d$. Moreover, the Hessian matrix $\Hess f(x_0)$ at   $x_0$ is positive definite. 
\medskip

\item []\textbf{[A$_{\infty}$]} All the derivatives of $f$ (resp. of $\ell $) of order larger or equal than $2$ (resp. than~$1$) are bounded over $\mathbf R^d$. In addition,  $f$ is  bounded from below and   there exist $C,R>0$  such that 
\begin{equation*}\label{eq.coerive2}
\forall |x|\geq R,\; |\nabla f(x)|\geq C.
\end{equation*}

\end{enumerate}

 Assumption {\rm \textbf{[A$\perp$]}}  is rather generic as  explained  in Section \ref{sec.Ap} below. Besides, the   non local assumptions on  the vector field $b$ made in  {\rm \textbf{[A$\perp$]}}, {\rm \textbf{[A$_{x_0}$]}}, and {\rm \textbf{[A$_{\infty}$]}} are unnecessary for our main result to hold, see the note just after Theorem \ref{th.1}.
  In all this work, we   always assume  that  {\rm \textbf{[A$\perp$]}}, {\rm \textbf{[A$_{x_0}$]}}, and {\rm \textbf{[A$_{\infty}$]}} are satisfied. 
 In addition, throughout this work,   $\Omega$ is a $\mathcal C^\infty$ bounded subdomain\footnote{We recall that by definition, a domain is a nonempty connected open set.} of $\mathbf R^d$ containing $x_0\in \Omega$. The quantity of interest in this work is the first exit time from $\Omega$ for the process \eqref{eq.langevin} which will be  denoted by $\tau_\Omega$, i.e.  $$\tau_\Omega:=\inf\{t\ge 0, X_t\notin \Omega\}.$$ 

Let us mention that  the vector field $b$ does not vanish over $\partial \Omega$. In this case,   $\partial \Omega$ is said to be   non characteristic  and the   so-called \textit{generalized saddle points} $z$ of $f$ on $\partial \Omega$ (see the set $\mathscr P_{\rm sp}$ defined \eqref{eq.PP}) will    play a crucial role in the asymptotic formula of  the mean exit time $\mathbf E[\tau_\Omega]$ from $\Omega$.  Such critical points are associated with edge shaped barriers (or reflected barriers) and were considered by Kramers in its celebrated work~\cite{Kra}, see also~\cite[p. 836-837]{MATK1982}.

   
   \subsection{Purpose of this work} 
\label{sec.purp}
 In this work, we derive  a new sharp asymptotic formula in the limit $h\to 0$ (and actually prove a complete asymptotic expansion in power of $h$)  of the mean exit time $\mathbf E[\tau_\Omega]$  from $\Omega$ for the process \eqref{eq.langevin} when $X_0=x\in \Omega$ (uniformly in the relevant  compact sets), see our main result below,  Theorem \ref{th.1}. 
 
 Contrary to previous contributions in this field (see Section \ref{sec.lib}), 
we do not assume the process  \eqref{eq.langevin}  to be Gibbsian, which is the main novelty 
of this work.
More precisely, we do not assume  that 
the Gibbs measure $$\mu_{\rm Gibbs}(dx):=e^{-\frac 2h f}dx$$ is  invariant for the process \eqref{eq.langevin} (where $dx$ is the Lebesgue measure  over $\mathbf R^d$),
or equivalently that  $\div(\ell) =0$ over $\mathbf R^d$.

This
has    a strong impact on the Eyring-Kramers formula we derive in Theorem~\ref{th.1} for the mean exit time from $\Omega$ for the process \eqref{eq.langevin}. Indeed, compared to the  Gibbsian case,  the  new (non local) terms $\exp [\int_ 0^{+\infty}  \div (\ell)(\psi_t(z))dt ]$, which are  attached to each generalized saddle point $z$  of $f$ over $\partial  \Omega$,  appear in the pre-factor as a measure of the non-Gibbsianness of the process, see \eqref{eq.EK-PF}. 
Let us emphasize that these 
terms can be greater or bigger than $1$, see the comments following 
Theorem~\ref{th.1}. 
Hence, in comparison with the Gibbsian case,  both an acceleration or a deceleration of the exit time from~$\Omega$ can occur.
A very similar formula   was already derived  in \cite{BoRe16} with formal computations when  the saddle points $z$ are  critical points of $f$. In this case, the difference appears in the fact that  the integral of $\div (\ell)(\psi_t(z))$ runs from $-\infty$ to $+\infty$. The    non-Gibbsianness of the  process  \eqref{eq.langevin} introduces several difficulties  in the analysis of the precise $h$-limit of the mean exit time from $\Omega$. 
This follows   in particular  from the fact that a sufficiently good approximation (at least in $\mathcal C^1$-norm) of the non explicit  invariant probability density~$p_h$ is  needed. 

In addition, the spectral analysis performed here is more involved than in our previous work in the Gibbsian case~\cite{peutrecCMP}, with 
the construction of a different and much more precise quasi-mode for the principal eigenfunction of the infinitesimal generator of \eqref{eq.langevin} than the one built in~\cite[Section 5.2]{peutrecCMP}.
As a by-product, to derive a sharp asymptotic formula in the limit $h\to 0$ for  $\mathbf E[\tau_\Omega]$,   we do not assume that the vector field $\ell$ vanishes  at each  generalized saddle point $z$ of $f$  on $\partial \Omega$
as it was
the case in \cite[Section 5.1]{peutrecCMP}. In particular, when the process
 \eqref{eq.langevin}   is  Gibbsian,   Theorem \ref{th.1} improves the statement of \cite[Theorem~2]{peutrecCMP} in the case where 
 the vector field $b$ does not vanish over $\partial \Omega$\footnote{Let us mention that this technical issue did not appear at a critical point.}.


Let us finally mention that  we do not neither assume $b\cdot n _{\Omega}<0$ over $\partial \Omega$, where $n _{\Omega}(z)$ denotes the unit outward normal  vector to $\partial \Omega$ at $z\in \partial \Omega$, as it
was the case  for technical reasons in  the classical  pioneering works dealing with  the exit event from a bounded domain when $h\to 0$, see e.g.~\cite{DeFr,day-83,day1987r} or~\cite[Chapter 4]{FrWe}.

\begin{remark}
This work opens up several possibilities, such as considering the case where $b$ has a critical saddle point on the boundary of $\Omega$. The behavior as $h\to 0$ of the invariant measure near such a point is tricky to derive. Such a situation  is left to a future work. 
\end{remark}

      \subsection{Direct consequences of Assumptions {\rm \textbf{[A$\perp$]}}, {\rm \textbf{[A$_{x_0}$]}}, and {\rm \textbf{[A$_{\infty}$]}}} 
\label{sec.As}
 In this section, we  give some  consequences of {\rm \textbf{[A$\perp$]}}, {\rm \textbf{[A$_{x_0}$]}}, and {\rm \textbf{[A$_{\infty}$]}} which will be used in this work.     On the one hand, note  that \textbf{[A$_{\infty}$]} implies that 
$f$ is coercive, i.e.
\begin{equation}\label{eq.coerive}
\text{$f(x)\to +\infty$ as $|x|\to +\infty$.}
\end{equation}
Furthermore,  \textbf{[A$\perp$]}  and \textbf{[A$_{\infty}$]} imply that for all $x\in \mathbf R^d$,  the curve~$t\mapsto \varphi_{t}(x)$ solution on $\mathbf R^d$ to 
\begin{equation}\label{eq.flow}
\frac{d}{dt}\varphi_{t}(x)  = b(\varphi_{t}(x)), \ \varphi_{0}(x)=x,
   \end{equation}
is well defined over $\mathbf R_+$. 
Moreover, when in addition  \textbf{[A$_{x_0}$]} is satisfied,  $x_0$ is the unique local minimum of the function $f$ and  it holds $\{z\in \mathbf R^d, b(z)=0\}=\{x_0\}$ (in particular $\ell (x_0)=0$). For all $x\in \mathbf R^d$,   the $\omega$-limit set  $\omega(x)$ of $x$ is reduced to $\{x_0\}$ (see~\cite[Section 1.3]{peutrecCMP}), i.e.   $$\omega(x)=\{x_0\}.$$
 In addition, thanks to
 \cite[Lemma 1.4]{BoLePMi22} (see also~\cite{landim2018metastability} for a similar result),  the matrix $\Jac b(x_0)=-^t(\Hess f(x_0)+^t\Jac\ell (x_0))  $ admits precisely $d$ eigenvalues which   all have a negative real part.    Finally,  the relation 
 $\ell\cdot\nabla f=0$ yields that the matrix $\Hess(f)(x_0)\Jac\ell (x_0)$ is antisymmetric and hence $\div \ell (x_0)=\tr(\Jac\ell (x_0))=0$.

 
  \subsection{Preliminary analysis and generalized saddle points} 
  
  \subsubsection{The domain $\Omega$  is a well of potential}
Recall that $\Omega$ is assumed to be a $\mathcal C^\infty$ bounded domain of $\R^{d}$ containing~$x_0$. Let us define 
\begin{equation*}
\label{eq.cmin}
\mathscr{C}_{{\rm min}}:=    
\Omega \cap \{f<\min_{\partial \Omega}f  \},
\end{equation*} 
where for  $\mu\in \mathbf R$,
we use the notation $  \{f< \mu\}:=\{x\in \mathbf R^d,\  f(x)< \mu\}$. The sets  $\{f\le  \mu\}$ and $\{f=\mu\}$ are defined similarly. 

When {\rm \textbf{[A$_{x_0}$]}} and \textbf{[A$_{\infty}$]} are satisfied, we are thus in
the following geometrical situation which shows that the domain $\Omega$  looks like a single well of the potential 
function $f:\mathbf R^d\to \mathbf R$.

\begin{lemma}\label{le.Geo}
Assume that  {\rm \textbf{[A$_{x_0}$]}}  and {\rm \textbf{[A$_{\infty}$]}} are satisfied. Then:
\begin{enumerate}
\item[{\rm\textbf{[a]}}] The function $f$ admits for sole global minimum point  $x_0$ in $\mathbf R^d$ and thus~in~$\overline \Omega$. 
\item[{\rm\textbf{[b]}}] The set  $\mathscr{C}_{{\rm min}}$ is equal to $\{ f<\min_{\partial \Omega}f \}$; it   contains $x_{0}$ and is  connected. In addition,   $\pa \mathscr{C}_{{\rm min}}\cap \partial \Omega \neq \emptyset$. 
\end{enumerate}
\end{lemma}





\begin{proof} The proof is elementary. 
Recall first that thanks to \eqref{eq.coerive} following from {\rm \textbf{[A$_{\infty}$]}}, $f$ admits a global minimum on $\R^{d}$, and using {\rm \textbf{[A$_{x_0}$]}}, 
it has for only global
minimum point $x_0$. This implies Item \textbf{[a]}. 

Notice moreover that $\mathscr{C}_{{\rm min}}$ contains $x_{0}$ and satisfies $\mathscr{C}_{{\rm min}}=\overline\Omega \cap \{f<\min_{\partial \Omega}f  \}$. It is thus nonempty and both open and closed in $\{ f<\min_{\partial \Omega}f \}$. 
Note also that  for any $\lambda >f(x_0)$, the nonempty (and relatively compact) open  set  $\{ f<\lambda\}$
is connected. 
In particular, the set $\{f<\min_{\partial \Omega}f  \}$ is connected, which implies that 
$$\mathscr{C}_{{\rm min}}=\{f<\min_{\partial \Omega}f  \}\quad\text{is connected}.$$
Finally, since $f:\R^{d}\to\R$ does not have any local minimum on  $\{ f=\min_{\partial \Omega}f \}$, it follows that 
$$\partial \mathscr{C}_{{\rm min}}\ =\ \{ f=\min_{\partial \Omega}f \},$$
which implies that $\partial \mathscr{C}_{{\rm min}}\cap \partial\Omega= \{ f=\min_{\partial \Omega}f \}\cap \partial\Omega \neq\emptyset$.
\end{proof}



   \subsubsection{Set of generalized saddle points of $f$}
 When  {\rm \textbf{[A$_{x_0}$]}}  and {\rm \textbf{[A$_{\infty}$]}} are satisfied, we  define  the (nonempty) set 
\begin{equation}\label{eq.PP}
\mathscr  P_{\rm sp}:= \pa \mathscr{C}_{{\rm min}}\cap \partial \Omega .
\end{equation}
Note  that every $z\in \mathscr  P_{\rm sp}$ is a global minimum of $f|_{\partial \Omega}$. A point  $z\in \mathscr  P_{\rm sp}$ is a  so-called  generalized saddle point. This is due to the fact that   
$\partial_{n_\Omega} f(z)>0$ and, for that reason, when the potential function $f$ is extended by $-\infty$ outside $\overline \Omega$ (notice that this  extension is the one which is compatible with the absorbing boundary condition on $\partial \Omega$),  the point $z$ is  geometrically a first-order saddle point of $f$.  
These points have  a crucial role in the asymptotic equivalents of the mean exit time from~$\Omega$, see indeed Theorem \ref{th.1}.  

Our last assumption is the following
$$\textbf{[A$_{\mathscr  P_{\rm sp}}$]}  \text{ For all $z\in  \mathscr  P_{\rm sp}$,  $\det\Hess(f|_{\partial \Omega})(z)\neq 0$.} 
$$
Observe that when  \textbf{[A$_{\mathscr  P_{\rm sp}}$]}  holds, 
the set $\mathscr  P_{\rm sp}$ has a finite number of elements.  
We say that Assumption \textbf{[A]} holds when all the  four assumptions above are satisfied.

\subsection{Main result} In this section, we state our main result which is Theorem \ref{th.1}.  We first  recall some notation. 
 For every $x\in \Omega$, on sets  $t_x:=\inf  \{t\ge 0, \ \varphi_{t}(x)\notin \Omega\}>0$ the first time the curve $\varphi_{t}(x)$ exits $\Omega$ and we define the  domain of attraction of a subset $F$ of $\Omega$  by 
$$
  \mathscr A_\Omega(F):=\big \{ x\in \Omega, \,t_x=+\infty \text{ and } \omega(x)\subset F \big \}.
$$ 
Note that  when  {\rm \textbf{[A$\perp$]}}, {\rm \textbf{[A$_{x_0}$]}}, and {\rm \textbf{[A$_{\infty}$]}} hold,  $
\mathscr{C}_{{\rm min}}\subset   \mathscr A_\Omega( \{x_0\})$. 

Throughout the paper, we shall say that 
a family of scalar $(a^h)_{h\in]0,1]}$ admits a $h$-classical expansion, if there exists a sequence $(a_n)_{n\in\N}$ such that for all $N\in\N$,     $a^h=\sum_{n=0}^N a_n h^n+O(h^{N+1})$.  Such an expansion  will be  denoted by   $a^h\sim\sum_{n\geq 0}a_nh^n$.
More generally, we shall say that a family of smooth functions  $(u^h)_{h\in]0,1]}$ on an open set $\Omega$
of $\R^{d}$ or of $\mathbf R_-^d:=\mathbf R^{d-1}\times \mathbf R_-$ admits a $h$-classical expansion on $\Omega$, if there exists a sequence $(u_n)_{n\in\N}$
of smooth functions on $\Omega$  such that for all compact $K\subset \Omega$, for all $k\in\N$, and for all $N\in\N$, $u^h=\sum_{n=0}^N u_n h^n+O(h^{N+1})$ in the $C^{k}(K)$ topology. This expansion is also  denoted $u^h\sim\sum_{n\geq 0}u_nh^n$ and when $u_n\equiv0$ for all $n\in\N$, we write $u^{h}=O(h^{\infty})$.

\begin{theorem}\label{th.1}
Assume that   {\rm \textbf{[A]}} holds. Then for every compact subset $K\in \mathscr A_{\Omega}(x_0)$, it holds uniformly in $x\in K$,
\begin{equation}\label{eq.EK}
\mathbf E_x[\tau_\Omega]= \kappa(h)\sqrt{ h} \, e^{\frac 2h (\min_{\pa \Omega}f -f(x_0))},
\end{equation}
where $\kappa(h)$ admits a $h$-classical expansion $\kappa(h)\sim \sum_{j\geq 0}\kappa_j h^j$ and 
\begin{equation}\label{eq.EK-PF}
\frac 1{\kappa_0}=\displaystyle{\frac {\sqrt{ \det\Hess f(x_0)} }{\sqrt \pi}  \sum_{\substack{z\in\mathscr  P_{\rm sp}}}  \,   \frac{\partial_{ n_\Omega}f(z) }{\sqrt{  \det\Hess f_{\vert\partial\Omega}(z)    } }\exp\Big[\int_ 0^{+\infty} \div (\ell)(\psi_t(z))dt\Big]},
\end{equation}
where $t\ge 0\mapsto \psi_t(x)$ is 
the   (global)  solution on $\mathbf R^d$ to 
 \begin{equation}\label{eq.solition}
\frac{d}{dt}\psi_{t}(x)  = - ( \nabla f- \ell )(\psi_{t}(x)), \  \psi_{0}(x)=x.
\end{equation} 
\end{theorem} 
Let us make two comments on this theorem.
\begin{enumerate}
\item[$\mathsf 1$.] 
Given a domain $\Omega$,  the  mean exit time $\mathbf E[\tau_\Omega]$  only depends on the killed process $(X_t,0\le t<\tau_\Omega)$ and thus only  on the values of $b$ in $\overline \Omega$ (roughly speaking, it   does not depend on the non local assumptions in {\rm \textbf{[A]}}). For that reason, the  Eyring-Kramers formula we derive in Theorem \ref{th.1} above does not depend on the values of $b$ outside $\overline \Omega$. 

\item[$\mathsf 2$.] It appears that for each $z\in\mathscr  P_{\rm sp}$, the term $\int_ 0^{+\infty} \div (\ell)(\psi_t(z))$, which in some sense
measures the non-Gibbsianness of the process when it is not~$0$, can be positive or negative.
Hence, compared with the reversible case $\ell=0$ or with the Gibbsian case $\div\ell=0$~\cite{peutrecCMP},  both can occur, acceleration or deceleration of the (mean) exit time from $\Omega$. 

To observe that both situations can occur, consider, in $\R^{2}$, the open disc
$\Omega=D((0,-1),2)$  of center $(0,-1)$ and radius $2$ and  the fields $b_{\pm}=-(\nabla f +\ell_{\pm})$, where $f(x)=\frac12|x|^{2}$ on $\R^{2}$ and $\ell_{\pm}(x)=\pm (x_{1}x_{2},-x^{2}_{1})$ on $\overline\Omega$.
In this setting, the assumption {\rm \textbf{[A]}} is satisfied with $x_{0}=0$ (once $\ell_{\pm}$ have been conveniently defined outside~$\overline\Omega$),  $\mathscr{C}_{{\rm min}}=D(0,1)$, and $\mathscr  P_{\rm sp}= \pa \mathscr{C}_{{\rm min}}\cap \partial \Omega=\{z\}$, where $z=(0,1)$.
In addition, we have  $\psi_t(z)=(0,e^{-t})$ for all $t\ge 0$ and $\div\ell_{\pm}(x)=\pm x_{2}$ for all $x\in \overline\Omega$.
It follows that $\int_ 0^{+\infty} \div (\ell_{\pm})(\psi_t(z))=\pm 1$ can be negative or positive. 
\end{enumerate}


\subsection{Related results}
\label{sec.lib}
In the boundary case, the first works on the mean exit time were probably those of Freidlin–Wentzell, where  in the one well setting, and when  $h\to 0$, the limit of $h\ln \mathbf E_x[\tau_{\Omega}]$   has been derived in~\cite[Chapter 4]{FrWe}  when $b\cdot n _{\Omega}<0$ over $\partial \Omega$. In this setting, it has   also been proved in~\cite{day-83} that   $ \mathbb E_x[\tau_{\Omega}]\lambda_{h}=1+o(1)$ when $h\to 0$, where $\lambda_h$ is the principal eigenvalue of (minus) the generator $L_h$ of \eqref{eq.langevin}, see Section \ref{sec.LL} below. Closely related results have also been derived in~\cite{DeFr,nectoux2017sharp} still  in this setting.   We also  refer to~\cite{MaSc,schuss2009theory} where asymptotic formulas for $\mathbf E[\tau_{\Omega}]$ when $h \to 0$ have been  obtained  through formal computations in different geometrical settings. 
 A comprehensive review of the literature until the 2010s on this topic can be found in~\cite{Ber}. 
Moreover, in the case when   $b\cdot n _{\Omega}\neq 0$  on $\partial \Omega$,   the asymptotic behavior of the solution to the parabolic equation $\partial_t u= L_{h}u $ has been studied in  \cite{ishii-souganidis-1,ishii-souganidis-2} (where the quasilinear case is also treated).   In the reversible case (i.e. when $\ell=0$), Eyring-Kramers type formulas have been derived   in~\cite{NectouxCPDE} when $f$ has critical saddle points on $\partial \Omega$, see also~\cite{mathieu1995spectra,sugiura2001,helffer-nier-06,LeNi,di-gesu-le-peutrec-lelievre-nectoux-16,DoNe2,lelievre2022eyring} for related results in the reversible case. 

Let us also mention that asymptotic estimates on eigenvalues and  on mean transition times in   the boundary less case have been widely studied. 
When the process is   reversible, we refer to~\cite{BGK,BEGK,Berglund_Gentz_MPRF,BD15,galves-olivieri-vares-87,miclo-95,avelin}. In the non-reversible case and  when the process \eqref{eq.langevin} is Gibbsian  (i.e. when $\div(\ell) =0$), sharp equivalent of mean transition times   have been derived in~\cite{seoARMA,lee2022non,LandimDrift2} (see also~\cite{landim2024metastability,lee2025eyring}), and   sharp asymptotic  formulas for the smallest eigenvalues of $L_h$ have been obtained in~\cite{peutrec2019sharp}  (see also~\cite{HeHiSj,BoLePMi22,bony2024real}).   Finally, as already mentioned,   the generalization  of  the Eyring–Kramers formula for mean transition times  to non Gibbsian  diffusion processes have been derived in~\cite{BoRe16} with formal calculations.

\section{Asymptotic behavior of the stationary distribution}



  \subsection{Asymptotic expansion of the stationary distribution}

In this section, we recall  and  improve a result of Sheu \cite{sheu1986} and Mikami~\cite{Mik90} (see also ~\cite{day1987r}) on the  properties of the stationary measure of \eqref{eq.langevin}, see Theorem \ref{th.sheu}.

For a measure $\nu$ over a subset $\mathscr M$ of $\mathbf R^d$, the set $H^k(\mathscr M,\nu(dx))$ stands for the usual (weighted) Sobolev space  of regularity $k\ge 0$ over $\mathscr M$ for the  measure $\nu(dx)$.  
 The infinitesimal generator of the diffusion \eqref{eq.langevin} is $\frac h2\Delta+b \cdot \nabla$ and will rather work  with minus this operator, namely with 
\begin{equation}\label{eq.Lh}
L_h=-\frac h2\Delta+ (\nabla f+\ell)\cdot \nabla.
\end{equation}
 The formal adjoint of $L_h$ in $L^2(\Omega, dx)$ is denoted by  $L_h^\star$. It is the operator acting on smooth function $u:\mathbf R^d\to \mathbf R$ as  
\begin{equation}\label{eq.Lstar}
L_h^\star u=-\frac h2 \Delta u + {\rm div}\, (b \, u  ).
 \end{equation}
 
For all $h>0$ fixed, the existence of  an invariant probability measure $\mu_h$ for the process \eqref{eq.langevin} follows from \cite[Lemma 1.2]{sheu1986} (note that according to the discussion in Section \ref{sec.As},  {\rm \textbf{[A$\perp$]}}, {\rm \textbf{[A$_{x_0}$]}}, and {\rm \textbf{[A$_{\infty}$]}} are  simply    a shorten formulation  Assumption (A) in~\cite{sheu1986}). 
It is well-known that since the vector field $b$ is smooth,   $\mu_h$ is the unique   invariant probability measure. This  follows from the fact that the process \eqref{eq.langevin}  is topologically irreducible and strongly Feller~\cite[Theorem 1.1]{peszat1995strong}. Moreover, $\mu_h$  has  a smooth density $p_h: \mathbf R^d\to \mathbf R_+$ w.r.t. the Lebesgue measure $dx$, see e.g. \cite{nualart2006malliavin,bogachev2022fokker},  which is positive over $\mathbf R^d$.  


\begin{theorem} \label{th.sheu}
 Assume   {\rm \textbf{[A$\perp$]}}, {\rm \textbf{[A$_{x_0}$]}}, and {\rm \textbf{[A$_{\infty}$]}}. 
Then, the positive function
$R^h$ defined by
\begin{equation}\label{eq.Rhh}
R^h: x\in \mathbf R^d\mapsto p_h(x)h^{d/2}e^{\frac 2h (f(x)-f(x_0))}
\end{equation}
 admits a real-valued $h$-classical expansion on $\R^{d}$. 
More precisely, 
we have 
$R^{h}\sim \sum_{k\geq0}h^kR_k$ for a sequence of real-valued functions $(R_k)_{k\in \N}\subset  C^\infty(\mathbf R^d)$
which satisfy
\begin{equation}\label{eq.jacobi}
-( \nabla f-\ell )\cdot \nabla R_0 + R_0 \div(\ell )=0, \, R_0(x_0)=c_0,
\end{equation} 
and 
for, any $k\geq 1$,
\begin{equation}\label{eq.jacobi2}
-( \nabla f-\ell )\cdot \nabla R_k + R_k \div(\ell ) =-\frac 12\Delta R_{k-1},
\end{equation} 
where $c_0>0$ is defined by  $c_0 \int_{\mathbf R^d}e^{-x\cdot \Hess f(x_0)x}dx=1$, that is by
\be\label{eq:calc0}
c_0=|\det \Hess f(x_0)|^{\frac 12}\pi^{-\frac d2}.
\ee
\end{theorem} 

We recall that $R^{h}\sim \sum_{k\geq0}h^kR_k$ means that
for every $n\in\N$, for every compact set $K$ of $\mathbf R^d$, and for every $\alpha\in\N^d$,
the smooth function  $W^h_{n}:=R^h - \sum_{k=0}^nh^kR_k$ satisfies
\be\label{eq:estw2}
 \sup_{x\in K}|\partial^\alpha W^h_{n} (x)|\leq C_{K,\alpha} h^{n+1}
 \ee
when $h\to0$, for some constant $C_{K,\alpha}$ independent of $h$.
As we shall see in the proof, the case when $|\alpha|=0$ is already known \cite{sheu1986,Mik90}. To perform our analysis we need a stronger control on the derivatives of the $W^h_{n}$'s. This is the reason we   extend  it to the case    $|\alpha|>1$.

\begin{proof}
  From \cite[Lemma 1.1]{sheu1986}, one deduces that  $f=\frac 12 V$ where $V$ is the so-called Freidlin–Wentzell  quasi-potential defined by 
 \bes
 V(x)= \frac 12 \inf_{\phi(0)=x,\;\phi(+\infty)=x_0}\int_0^\infty|\dot\phi(s)+b(\phi(s))|^2ds.
 \ees
 In particular, the function  $V$ is $C^\infty$ on $\R^d$. In addition, Assumptions (A.2), (A.3), and (A.4.r) in \cite{Mik90} are satisfied. Hence, by  \cite[Theorem 1.3]{Mik90}, we get exactly    all the assertions of the theorem   except   the estimates 
 \eqref{eq:estw2} when $|\alpha|\geq 1$, which we  prove now (observe indeed  that \eqref{eq:estw2} with $|\alpha|=0$ is proved in \cite{Mik90}).

  In the following $C>0$ denotes a constant independent of  $h>0$ that may change from one occurence to another. The Sobolev space $H^k(\mathbf R^d, dx)$ is simply denoted by $H^k$ and its norm by $\Vert \cdot \Vert_{H^{k}}$. The scalar product in $L^2(\mathbf R^d, dx)$ is denoted by $\langle\cdot,\cdot \rangle _{L^2}$. 
For ease of notation, we set $b^*=2(\nabla f-\ell )$. Then, Equations  \eqref{eq.jacobi} and  \eqref{eq.jacobi2}
  rewrite 
 \bes
 b^*\cdot \nabla R_k-2\div(\ell )\, R_k-\Delta R_{k-1}=0
 \ees
for all $k\geq 0$, with the convention $R_{-1}=0$.  On the other hand, since $p_h$ is the invariant  density of the process \eqref{eq.langevin},  it follows that  $L_h^\star \, p_h=0$ (see \eqref{eq.Lstar}). Hence, we deduce that 
\bes
h\Delta R^h-b^*\cdot\nabla R^h+2\div(\ell ) R^h=0.
\ees
Combining these two identities, we get for every $n\in\N$,
\begin{equation}\label{eq:pfthsheu1}
(h\Delta -b^*\cdot\nabla + 2\div(\ell ))W_n^h=-h^{n+1}\Delta R_{n}.
\end{equation}
By Sobolev embedding, in order to prove \eqref{eq:estw2}, it is sufficient to prove that  for any $\chi\in C_c^\infty(\R^d)$ and any $\alpha\in\N^d$, there exists $C>0$, such that  
\be\label{toproveWnh}
\Vert \chi \partial^\alpha W_n^h\Vert_{L^2}\leq C\,  h^{n+1}.
\ee 
To prove \eqref{toproveWnh}, we first  show the following   \textit{a priori}  estimate:
\be\label{toproveWnhweak}
\forall \alpha\in\N^d, \forall \chi\in C_c^\infty(\R^d),\;\exists C>0,\;\Vert \chi \partial^\alpha W_n^h\Vert_{L^2}\leq Ch^{n+1-\frac {|\alpha|} 2}.
\ee
We prove this estimate by   induction on   $s=|\alpha|$. 
To lighten the notation, the function $\partial^\alpha W_n^h$ will be simply denoted by $w^\alpha$.   As  explained above, the case $s=0$ holds true since it is proved in~\cite{Mik90}.    Let us now assume that the estimate  \eqref{toproveWnhweak} is true for all $s'\le s$, for some $s\ge 0$. Let $\chi\in C_c^\infty(\R^d)$ be a real-valued function and let $\alpha\in\N^d$ be such that $|\alpha|=s$. 
Differentiating  $\alpha$ times Equation \eqref{eq:pfthsheu1}, $w^\alpha$ satisfies
\be\label{eq:pfthsheu1deriv}
(h\Delta -b^*\cdot\nabla)w^\alpha=-h^{n+1}\partial^\alpha\Delta R_{n}+\sum_{\beta, |\beta|\leq|\alpha|}a_{\alpha,\beta}w^\beta,
\ee
where the functions $a_{\alpha,\beta}$, which are   linear combinations of derivatives of $\div(\ell)$ and $b^*$,  are smooth functions over $\R^d$ and are  independent of $h$.
Since the $a_{\alpha,\beta}$ are bounded on $\supp(\chi)$,   using   the Cauchy–Schwarz inequality,  one deduces that
\bes
|\<(h\Delta-b^*\cdot\nabla)w^\alpha,\chi^2w^\alpha\>_{L^2}|\leq C(\sum_{\beta, |\beta|\leq|\alpha|}\Vert \chi w^\beta\Vert_{L^2}+h^{n+1}\Vert \chi\partial^\alpha \Delta R_n\Vert_{L^2})\Vert \chi w^\alpha\Vert_{L^2}.
\ees
Since all the derivatives of $\Delta R_n$ are bounded on any compact set and using the induction hypothesis, this implies 
\bes
|\<(h\Delta-b^*\cdot\nabla)w^\alpha,\chi^2 w^\alpha\>_{L^2}|\leq C h^{2n+2-s},
\ees
and hence
\be\label{eq:pfWnh1}
|\<h\Delta w^\alpha,\chi^2w^\alpha\>_{L^2}|\leq |\<b^*\cdot\nabla w^\alpha,\chi^2w^\alpha\>_{L^2}|+ Ch^{2n+2-s}.
\ee
On the other hand, integrating by parts, we get 
\bes
\begin{split}
\<b^*\cdot\nabla w^\alpha,\chi^2w^\alpha\>_{L^2}=&-\<b^*\cdot\nabla w^\alpha,\chi^2w^\alpha\>_{L^2}
-\<\div(b^* \chi^2) w^\alpha, w^\alpha\>_{L^2},
\end{split}\ees
so that $2\<b^*\cdot\nabla w^\alpha,\chi^2w^\alpha\>_{L^2}=-\<\div(b^* \chi^2) w^\alpha, w^\alpha\>_{L^2}$. Since $\div(b^* \chi^2) $ has a compact support, 
using again the induction hypothesis, it follows that
\bes
|\<b^*\cdot\nabla w^\alpha,\chi^2w^\alpha\>_{L^2}|\leq Ch^{2n+2-s}.
\ees
Combined with \eqref{eq:pfWnh1},  this finally yields
\be\label{eq:pfWnh2} 
|\<\Delta w^\alpha,\chi^2w^\alpha\>_{L^2}|\leq C h^{2n+2-s-1}.
\ee
On the other hand, by successive integration by parts, one gets
\bes
\begin{split}
-\<\Delta w^\alpha,\chi^2 w^\alpha\>_{L^2}&=\<\nabla w^\alpha,\chi\nabla\chi \, w^\alpha\>_{L^2}+\<\chi\nabla w^\alpha,\nabla(\chi w^\alpha)\>_{L^2}\\
&=\<\nabla w^\alpha,\chi\nabla\chi w^\alpha\>_{L^2}+\Vert \nabla(\chi w^\alpha)\Vert_{L^2}^2-\<\nabla\chi w^\alpha,\nabla(\chi w^\alpha)\>_{L^2}\\
&=\Vert \nabla(\chi w^\alpha)\Vert_{L^2}^2-\Vert \nabla(\chi) w^\alpha\Vert_{L^2}^2.
\end{split}
\ees
Combining this previous  identity with \eqref{eq:pfWnh2} and using once again the induction hypothesis, we deduce that
$
\Vert \nabla(\chi w^\alpha)\Vert_{L^2}^2=O(h^{2n+2-s-1})$. Therefore, it holds: 
\be\label{eq:pfWnh3}
\Vert  \chi  \nabla w^\alpha\Vert_{L^2}=O(h^{n+1-\frac{s+1}2}).
\ee
This  proves the estimate \eqref{toproveWnhweak} at rank $s+1$. Let us now improve this bound, and more precisely let us show   \eqref{toproveWnh}. We observe that for any $n\leq n'\in\N$, one has
\bes
\chi \partial^\alpha W_n^h=\chi \partial^\alpha W_{n'}^h+\sum_{k=n+1}^{n'} h^k\chi \partial^\alpha R_k.
\ees
Consequently, since the $R_k$ are $C^\infty$ and independent of $h$, for any $s\in\N$, one deduces from \eqref{toproveWnhweak} that
\bes
\Vert  \chi\partial^\alpha  W_n^h\Vert_{L^2}\leq \Vert  \chi \partial^\alpha W_{n'}^h\Vert_{L^2}+\sum_{k=n+1}^{n'} h^k\Vert \chi \partial^\alpha  R_k\Vert_{L^2}\leq C_s h^{n'+1-\frac{|\alpha|}2} +Ch^{n+1}
\ees
We finally obtain \eqref{toproveWnh} by taking $n'=n+\lceil \frac{|\alpha|}2\rceil$, which
completes the proof.  
\end{proof}


We now discuss the term $R_0$ which will play a crucial role in the Eyring-Kramers formula for the mean exit time. 
As already noticed in \cite[Theorem 3.1]{sheu1986} (see also \cite{day1987r}), it is standard to show that  $R_0$ actually admits the expression  \begin{equation}\label{eq.R0}
R_0(x)= c_0\exp\Big[\int_ 0^{+\infty} \div ( \ell ) (\psi_t(x))dt\Big], \ \forall x\in \mathbf R^d, 
\end{equation} 
 where we recall that~$t\mapsto \psi_{t}(x)$ is the global solution to  \eqref{eq.solition}. 
 Indeed,   using \eqref{eq.jacobi}, it holds
$$\frac{d}{dt} R_0(\psi_{t}(x))+ \div (\ell) (\psi_t(x)) R_0(\psi_{t}(x))=0, \forall t\ge 0.$$
Hence,   for $t\ge 0$, $R_0(\psi_{t}(x))=R_0(x)\exp[-\int_ 0^{t} \div( \ell )(\psi_s(x))ds ]$. The identity \eqref{eq.R0}  then follows taking the limit $t\to+\infty$. Note that  $ \div ( \ell )(\psi_s(x))\to 0$ exponentially fast as $s\to +\infty$  since   
 $\psi_s(x)\to x_0$  exponentially fast as $s\to +\infty$ ($x_{0}$ being non-degenerate), $\div(\ell )$ is globally Lipschitz,  and  $\div(\ell )(x_0)=0$.
  Note also that the formula \eqref{eq.R0} implies that $R_0(x)$  is positive over $\R^d$.
  



 \section{Spectral point of view} 
  
    In the two next sections we recall some results from  \cite{peutrecCMP} that will be used to prove our main result. We first note that since $p_h$ is smooth and positive over $\mathbf R^d$ and because  $\Omega$ is bounded, it holds for every $h>0$ and $k\ge 0$, 
 \begin{equation}\label{eq.Hk}
 H^k(\Omega,e^{-\frac 2h f}dx)=H^k(\Omega,dx)=H^k(\Omega,p_hdx),
  \end{equation}
 and all the involved norms are equivalent (with constants depending on $h>0$).  In the following, the norm (resp. the scalar product) of the space  $L^2(\Omega, \nu(dx))$ is denoted by $ \Vert \cdot \Vert_{L^2(\nu(dx))}$ (resp. $\langle \cdot,\cdot \rangle_{L^2(\nu(dx))}$) and 
 we write $H^k(\Omega,\nu(dx))$ to indicate that the ambient space under consideration is 
 the space $L^2(\Omega,\nu(dx))$.  
 
\subsection{Spectral results at fixed $h>0$}
 \label{sec.LL}
In \cite[Proposition 3]{peutrecCMP}, we proved that the operator 
$L_h$  with domain   $\mathscr D_1=H^2(\Omega, e^{-\frac 2h f}dx)\cap H^1_0(\Omega,e^{-\frac 2h f}dx)$  
  had the following spectral properties at $h>0$ fixed,  which are quite standard for an elliptic operator over a bounded domain:
 \begin{enumerate}
 \item[] {$\mathfrak 1$.}  The operator  $L_h|_{\mathscr D_1}$ is  maximal quasi-accretive and  has a compact resolvent. Its spectrum is thus discrete.  
 \item[]{$\mathfrak 2$.}  The operator  $L_h|_{\mathscr D_1}$   has a principal  eigenvalue $\lambda_h\in \mathbf R_+^*$, i.e.  $\lambda _h$ 
 has algebraic multiplicity one, and   $\Re \mu >\lambda_h$ for every $\mu \in \sigma(L_h)\setminus \{\lambda_h\}$.  
 
 In addition, any associated eigenfunction $u_h$ has a sign in $\Omega$.  The normalized positive one  is called the principal eigenfunction.
  \end{enumerate}
   
On the other hand, we also notice that for all real function $u\in \mathscr D^\infty_0:=\{g \in \mathcal C^\infty(\overline \Omega), g=0 \text{ on } \partial \Omega\}$,
 $$\int_\Omega L_h u^2(x)\, p_h(x)\, dx=0,$$
  stemming from the identity $L_h^\star\,  p_h=0$ together with the fact that $\nabla u^2\cdot n_\Omega=0$ on $\partial \Omega$. Consequently, we have for any such real function $u$ the identity 
  \begin{equation}\label{eq.quadra}
 \langle u,L_hu\rangle_{L^2(p_hdx)}=  \int_\Omega u(x)L_h u(x)\, p_h(x)\, dx= \frac h2 \int_\Omega |\nabla u(x)|^2 \, p_h(x)\, dx.
   \end{equation}
   As simple as it may seem, this formula is a key ingredient in our analysis. Moreover, the gradient structure \eqref{eq.quadra} strongly suggests to rather work with  the operator 
$L_h$  with domain $\mathscr D_2=H^2(\Omega, p_hdx)\cap H^1_0(\Omega, p_hdx)$.  In view of \eqref{eq.Hk}, one can also consider $L_h$ with domain  $\mathscr D_3=H^2(\Omega, dx)\cap H^1_0(\Omega,dx)$. 
It is actually  easy to see that the spectrum of   $L_h$  is the same on each domain  $\mathscr D_i$, $i=1,2,3$ (associated with their respective scalar product), as well as both   the algebraic and geometric   multiplicities   of an eigenvalue. 
In addition, Items $\mathfrak 1$   and $\mathfrak 2$ above are satisfied for  each $L_h|_{\mathscr D_i}$, $i=1,2,3$, with the same principal eigenvalue $\lambda_h$ and associated eigenspace ${\rm Span}(u_{h}) $.
 Note also that the identity \eqref{eq.quadra} extends by density  to every real $u\in \mathscr D_i$, $i=1,2,3$.

In the following, for all $h>0$, we   choose $u_h$ such that  
$$u_h>0 \text{ in } \Omega \text{ and } \int_\Omega |u_h|^2p_h=1.$$


     \subsection{Spectral results when  $h\to 0$}

In this section, we recall  the following result which will be the starting point of  the proof of Theorem \ref{th.1}.

 \begin{theorem}[\cite{peutrecCMP}]\label{th.CMP}
 Assume   {\rm \textbf{[A$\perp$]}}, {\rm \textbf{[A$_{x_0}$]}}, and {\rm \textbf{[A$_{\infty}$]}}. 
Then, there exists  $c_1>0 $  such that, for all $c_2\in (0,c_1)$, there exist $h_0>0$ and $C>0$ such that, for all 
$\mathsf  z\in \{\mathsf z\in \mathbb C, \Re \mathsf z \le c_1 , \vert \mathsf z\vert \ge c_3\}$ and $h\in (0,h_0]$, 
\begin{equation}\label{eq.info0}
L_{h}-\mathsf  z     \text{  is invertible and } \ \  \Vert (  L_h-\mathsf z)^{-1}\Vert_{L^2(p_hdx)}   \le C.
 \end{equation} 
In addition, there exist $c>0$  and $h_0>0$ such that  for all $h\in (0,h_0]$, 
\begin{equation}\label{eq.info}
\sigma(L_{h})\cap \big\{ \mathsf z\in \mathbb C, \Re \mathsf z\le c\big\}=\{\lambda_{h} \big\}\  \text{ and }\ 
 \lim_{h\to 0}h\ln \lambda_{h}= -2\,\big(\min_{\pa \Omega}f -f(x_0)\big).
 \end{equation} 
Moreover, for every compact subset  $K$  of   $\mathscr A( \{x_0\})$,  there exist $c>0$  and $h_0>0$ such that  for all $h\in (0,h_0]$ and $x\in K$,
\begin{equation}\label{eq.ex} 
 \lambda_h  \mathbf  E_x[\tau_{\Omega}]=  (1+ O( e^{-\frac{c}{h}}) )   \text{    uniformly in $x\in K$}.
 \end{equation}
  \end{theorem}

  \noindent
  \textbf{Note}. It has also been proved in \cite{peutrecCMP} that the law of $\lambda_h\tau_{\Omega}(x)$  
converges exponentially fast to the exponential law of mean $1$, uniformly in $x$ in the compact subsets of  $\mathscr A( \{x_0\})$.

  \begin{proof}
  All the statements of Theorem \ref{th.CMP} have been proved in~\cite{peutrecCMP}, see indeed  Theorems 1 and  4 there\footnote{  In the setting of~\cite{peutrecCMP},  $\Omega$ is subdomain of the $d$-dimensional torus. It is actually straightforward to see that the assertions~\cite[Theorems  1 and 4]{peutrecCMP} are indeed still valid under {\rm \textbf{[A$\perp$]}}, {\rm \textbf{[A$_{x_0}$]}}, and {\rm \textbf{[A$_{\infty}$]}.} }. 
  Let us mention that concerning the second estimate in \eqref{eq.info0}, we actually proved
   in~\cite{peutrecCMP}
   that for such complex numbers $\mathsf z$ and for all $h>0$ small enough, 
  $$\Vert (  L_h-\mathsf z)^{-1}\Vert_{L^2(e^{-\frac 2h f}dx)}   \le C,$$ 
  namely that for all $\phi \in L^2(\Omega, e^{-\frac 2h f}dx)$, $\int_\Omega  |(  L_h-\mathsf z)^{-1} \phi|^2 e^{-\frac 2h f} \le C^2\int_\Omega  | \phi|^2 e^{-\frac 2h f}$, or equivalently  (see \eqref{eq.Rhh} and Theorem \ref{th.sheu})
  $$ \int_\Omega  |(  L_h-\mathsf z)^{-1} \phi|^2 \frac{p_h}{R_0+o(1)} \le C^2\int_\Omega  | \phi|^2 \frac{p_h}{R_0+o(1)},$$
  where $o(1)$ is uniform over $\overline \Omega$. 
  The fact that $\Vert (  L_h-\mathsf z)^{-1}\Vert_{L^2(p_hdx)}=O(1)$ then follows noticing that over $\Omega$, there exists $r_1,r_2>0$ such that 
  $r_1\le R_0\le r_2$. 
  \end{proof}

  In the next section we state the spectral counterpart of Theorem \ref{th.1}.
  \subsection{Sharp asymptotics of the principal eigenvalue}

  The following theorem provides the sharp equivalent of $ \lambda_h$ in the limit $h\to 0$.

  \begin{theorem}\label{th:sharpeigenval}
 Assume that   {\rm \textbf{[A]}} holds. Then,  there exists $h_0>0$ and a sequence $(\zeta_j)_{j\in\N}$ of real numbers such that 
 for all $h\in]0,h_0]$, one has
 \be\label{eq:asymptvp}
 \lambda_h=h^{-\frac 12}\zeta(h)e^{-\frac 2h (\min_{\pa \Omega}f -f(x_0))},
 \ee
 where $\zeta(h)$ admits a $h$-classical expansion $\zeta(h)\sim\sum_{j\geq 0}\zeta_j h^j$
 with 
 \be\label{eq:defzeta0}
 \zeta_0=\displaystyle{\frac {\sqrt{ \det\Hess f(x_0)} }{\sqrt \pi}  \sum_{\substack{z\in\mathscr  P_{\rm sp}}}  \,   \frac{\partial_{ n_\Omega}f(z) }{\sqrt{  \det\Hess f_{\vert\partial\Omega}(z)    } }\exp\Big[\int_ 0^{+\infty} \div (\ell )(\psi_t(z))dt\Big]},
 \ee
and  where we recall that $t\ge 0\mapsto \psi_t(x)$ is the solution to \eqref{eq.solition}.
  \end{theorem}

Note that combining Theorem \ref{th:sharpeigenval} with \eqref{eq.ex} yields the assertion of Theorem \ref{th.1}. The rest of this work is thus dedicated to the proof of Theorem \ref{th:sharpeigenval}.


  \section{Construction of an accurate quasi-mode for $u_h$}

  
  The goal of this section is to construct a very precise quasi-mode $u_h^{\rm{app}}$ for the principal eigenvalue $\lambda_h$ of $L_h$ in $L^2(\Omega, p_hdx)$, namely a   function approximating the principal eigenfunction $u_h$
  sufficiently well so that we can compute asymptotically $\lambda_h$ as $h\to 0$. The conditions on $u_h^{\rm{app}}$ are listed in Proposition \ref{prop:calculQM} below.  
  Roughly speaking, we want to choose the function   $u_h^{\rm{app}}$ equal to $1$ on a very large part of $\mathscr{C}_{min}$ but satisfying the boundary condition $u_h^{\rm{app}}=0$ on $\partial\Omega$. 
 This requires the construction of a quasi-mode $u_h^{\rm{app}}$ realizing the appropriate transition from $1$ to $0$ around $\partial\Omega$.
  The delicate part of  this construction  occurs around $\mathscr  P_{\rm sp}=\pa \mathscr{C}_{{\rm min}}\cap \pa \Omega$ (see \eqref{eq.PP}),  where we use  suitable local coordinates near each $z\in \mathscr  P_{\rm sp}$. 
 
 From now on, we assume \textbf{[A]}.
 
 \subsection{Local coordinates near $\mathscr  P_{\rm sp}$}\label{sec:coord}
 It turns out that the system of coordinates introduced in \cite{helffer-nier-06}  near such points 
 is well appropriate for  defining $u_h^{\rm{app}}$ and the upcoming computations.
Recall that $\mathscr  P_{\rm sp}=\pa \mathscr{C}_{{\rm min}}\cap \pa \Omega\neq \emptyset$ (see \eqref{eq.PP}) and that  thanks to  \textbf{[A$_{\mathscr  P_{\rm sp}}$]},  $\mathscr  P_{\rm sp}$ has a finite cardinality. 
In the following, we consider   $z\in \mathscr  P_{\rm sp}$.  
\begin{sloppypar}
Then, there exists  a neighborhood~$\mathscr U_z$   of $z$ in~$\overline \Omega$ and  a   coordinate system 
\begin{equation}\label{eq.cv-pa-omega-nablafnon0}
x\in \mathscr U_z \mapsto v=(v',v_d)=(v_1,\ldots,v_{d-1},v_{d})\in \mathbf R^{d}_-=\mathbf R^{d-1}\times \mathbf R_-
\end{equation}
 such that 
\begin{equation}\label{eq.cv-pa-omega-nablafnon02}
v(z)=0, \ \ \{x\in \mathscr U_z, \, v_d(x)<0\}=  \Omega\cap  \mathscr U_z, \ \{x\in \mathscr U_z,\,  v_d(x)=0\}=\pa \Omega \cap \mathscr U_z,
\end{equation}
 and
\begin{equation*}
\forall i,j\in\{1,\dots,d\},\ \ \ g_z\Big (\frac{\pa}{\pa v_i}(z),\frac{\pa}{\pa v_j}(z)\Big )=\delta_{ij}
\quad\text{and}
\quad
\frac{\pa}{\pa v_d}(z) = n_\Omega(z),
 \end{equation*}
where $ g_z$ is the metric tensor in the new coordinates.
We denote by $G=(G_{ij})_{1\leq i,j\leq d}$ its matrix, by $G^{-1}=(G^{ij})_{1\leq i,j\leq d}$ the inverse of $G$, and by $|G|=\det G$ its determinant. We also denote the canonical basis of $\mathbf R^{d}$ by  $(e_1,\dots, e_d)$. Then, defining $J:=\jac\, v^{-1}$, we have:  
\begin{equation}
\label{eq.G1}
G=\, ^t J J\,,\ \ G(0) =(\delta_{ij})\ \ \text{i.e.}\ \ {}^t J(0)=J^{-1}(0)\,,\ \ \text{and}\ \ n_\Omega(z) = J(0)e_d\,.
 \end{equation}
Let us now  determine    the operator $L_h$ in the above coordinates, see \eqref{eq.Lh}. Throughout this work, for any function $u$ defined on $\mathscr U_z$, we denote 
$\hat u=u\circ v^{-1}$ where $v$ is the above change of coordinates. For any $u\in C^\infty( \mathscr U_z)$, we have 
 $\widehat {L_h u}=\hat L_h\hat u$ with 
 \be\label{eq:Genercoord}
 \hat L_h=- \frac h{2\sqrt{|G|}} \div\circ  \sqrt{|G|} \,G^{-1}\circ\nabla+(G^{-1}\nabla \hat f+J^{-1}\hat \ell)\cdot\nabla.
 \ee
We can write  this operator as follows:
 \be\label{eq:Genercoord2}
 \hat L_h= -\frac h{2} \div\circ \, G^{-1}\circ\nabla- (b^\circ+h \rho^\circ)\cdot\nabla\,,
 \ee
 where $b^\circ$  and $\rho^\circ$   are the following smooth vector fields over $\mathscr U_z$:
 \be\label{eq:defbtilde}
b^\circ=-G^{-1}\nabla \hat f-J^{-1}\hat\ell \  \text{ and } \ \rho^\circ= \frac  1{2\sqrt{|G|}} \nabla (\sqrt{|G|}) \, G^{-1}.
\ee
  Let us now introduce the notation   $\ell^\circ=J^{-1}\hat\ell$, so that $b^\circ$ rewrites 
  \be\label{eq:defbtilde1}
b^\circ=-G^{-1}\nabla \hat f-\ell^\circ.
\ee
In addition,  according for example to~\cite[Section 3.4]{helffer-nier-06} (see also \cite{nectoux2017sharp}), 
the $v$-coordinates can be chosen such that:
\begin{equation}\label{eq.cv-pa-omega-nablafnon03}
\hat f(v',v_d)=f(z)+\mu_z v_d+ \frac 12 \,v' \cdot H_z v',
\end{equation} 
where 
\be\label{eq:defmuHz}
\mu_z:=\partial_{n_{\Omega}}f(z)>0\ \ \text{ and }\ \ H_z:=\Hess f|_{\partial\Omega}(z).
\ee
 Moreover, thanks to \textbf{[A$_{\mathscr  P_{\rm sp}}$]},    $0$ is   a non degenerate  (global) minimum of $\hat f|_{\{v_d=0\}}$.
  For $\delta_1>0$ and $\delta_2>0$ small enough, one finally defines  
  the following neighborhood of $z$  in $\overline \Omega$  (see~\eqref{eq.cv-pa-omega-nablafnon0}-\eqref{eq.cv-pa-omega-nablafnon02}),
\begin{equation}\label{eq.vois-11-pc}
\mathscr U^{\delta}_z=\big \{x\in \mathscr U_z,   \vert v'(x)\vert \le \delta_2  \text{ and } v_d(x)\in  [-2\delta_1, 0]\big  \},  \ \delta=(\delta_1,\delta_2).
\end{equation}
The set defined in  \eqref{eq.vois-11-pc}   is a cylinder centered at~$z$ in the $v$-coordinates. The parameters $\delta_1,\delta_2>0$ will be reduced a finite number of times to ensure several properties needed in Section \ref{prop:calculQM} to perform computations. Recall also that $f(z)=\min_{\pa \Omega}f>f(x_0)$. 
Then, up to choosing $\delta_1>0$ and $\delta_2>0$ smaller, we can assume that 
\begin{enumerate}  
 \item []\textbf{[C$^\delta_1$]}  The sets 
 $\mathscr U^{\delta}_z$, $z\in \mathscr  P_{\rm sp}$, are 
   pairwise disjoint, so in particular 
   $$ \argmin\limits_{ \mathscr U^{\delta}_z \cap \partial \Omega}f=\{z\}\,.$$
  \item []\textbf{[C$^\delta_2$]}   $\min_{\mathscr U^{\delta}_z}f>f(x_0)$,  so in particular    $x_0\notin \mathscr U^{\delta}_z$.
    \end{enumerate}  
Finally, according to \textbf{[C$^\delta_1$]} and using a continuity argument, once $\delta_2>0$ is fixed, 
one can choose  $\delta_1>0$ small enough  such that   
\begin{enumerate}   
    \item []\textbf{[C$^\delta_3$]}  There exists $r>0$  such that
$$
 \big \{x\in \mathscr U_z,   \vert v'(x)\vert =\delta_2  \text{ and } v_d(x)\in  [-2\delta_1, 0]\big  \} \subset\{f\ge f(z)+r\}.  
 $$ 
    \end{enumerate}  
We refer to Figure \ref{fig:S} for a schematic representations of the sets $\mathscr U^{\delta}_z$.  
\end{sloppypar}

\subsection{General form of the quasi-mode near $\mathscr  P_{\rm sp} $}
 Let $\chi\in C^\infty({\mathbf R}_{-},[0,1])$ be  a cut-off  function    such that (see \eqref{eq:defmuHz}):
\begin{equation}\label{eq.chi-cut}
\text{supp } \chi\subset  \Big [- \frac {\delta_1}{2}\mu_z ,  0\Big ] \, \text{ and } \,   \chi=1 \text{ on } \Big [-\frac{\delta_1}{4}\mu_z, 0\Big].
  \end{equation} 
  Define now, for each $z\in \mathscr  P_{\rm sp}$, the cylinder  $\mathscr V^{\delta}_z$  by 
\be\label{eq:defVz}
\mathscr V^{\delta}_z:=v\big (\mathscr U^{\delta}_z\big)=\{v=(v',v_d)\in\R^d,\;|v'|\le \delta_2,\;-2\delta_1\le v_d\le 0\}.
\ee
The set $\mathscr V^{\delta}_z$ is a neighborhood of $0$ in $\mathbf R_-^d$. 
For every $z\in \mathscr  P_{\rm sp}$, we look for a quasi-mode $u_h^{\rm{app}}$  defined on the cylinder $\mathscr U^{\delta}_z$ 
by
 \begin{equation}\label{eq.qm-local1x}
\forall x \in   \mathscr U^{\delta}_z, \  u_h^{\rm{app}}(x):= \varphi_{z} (v(x)) \
  \end{equation}
with a function $\varphi_z$  defined on the set $\mathscr V^{\delta}_z$  
by
 \begin{equation} \label{eq.qmPhi}
\forall v  \in  \mathscr V^{\delta}_z, \  \  \varphi_{z} (v):=\frac1{ N_{z,h}} \int_{\xi^z(v,h)}^0\chi(t)e^{\frac t h}  dt,
 \end{equation}
where for every $h\in]0,1]$, $v\in \mathscr V^{\delta}_z\mapsto\xi^z(v,h)$ is  a real nonpositive smooth function
 which will be constructed later, and $N_{z,h}$ is the normalizing constant
\be\label{eq:estimNhz}
N_{z,h}\ :=\ \int_{-\infty}^0\chi(t)e^{\frac t h}  dt\ =\ h+O(e^{-\frac ch}).
\ee 
We now turn to the  construction of an appropriate function $\xi^z$
vanishing on $\{v_{d}=0\}$ and
 such that $u_h^{\rm{app}}$
 satisfies  the Dirichlet boundary condition on $\partial \Omega$.

 \subsection{Construction of the function  $\xi^z$}
 
 We begin this section by deriving equations that shall satisfy  $\xi^z$ in order to make $u_h^{\rm{app}}$ sufficiently close   to the principal eigenfunction $u_h$ near each $z$. 
To this end, let us fix  $z\in  \mathscr  P_{\rm sp}$. 


 Since $\lambda_h$ is exponentially small,
we look for 
a function $\xi^z$ such that 
$u_h^{\rm{app}}$ is
an approximate solution of
 $L_hu_h^{\rm{app}}=0$.
More precisely, we look for a smooth function $\xi^z$ admitting a $h$-classical expansion $\xi^{z}\sim\sum_{j\geq 0} h^j\xi^z_j$
in $\mathscr V^{\delta}_z$
such that  $\xi^z_0\not\equiv0$, $\xi^z$
vanishes on $\{v_{d}=0\}$, and  $L_hu_h^{\rm{app}}(x)=O(h^\infty)$ in $\mathscr U^\delta_z$.
The latter relation reads in the $v$-coordinates 
 $$\hat L_h\varphi_z=O(h^\infty).$$ 
By \eqref{eq.qmPhi},  one has $\nabla \varphi_z=-  {\chi(\xi^z)e^{\xi^z/h} \nabla \xi^z }/{N_{z,h}} $. Consequently,  using   \eqref{eq:Genercoord2}, one has  
\begin{align}
 \nonumber
 \hat L_h \varphi_z &=-\frac h{2} \div( G^{-1} \nabla \varphi_z )- (b^\circ+h\rho^\circ)\cdot \nabla \varphi_z \\
 \nonumber
&=\frac {\chi(\xi^z)e^{\xi^z/h}}{N_{z,h}}\Big[\frac 12(G^{-1}\nabla\xi^z)\cdot \nabla\xi^z +\big[b^\circ\cdot\nabla\xi^z+h\big (\rho^\circ \cdot\nabla\xi^z+ \frac 12\div(G^{-1}\nabla\xi^z)\big ) \big ]  \Big]\\
 \label{eq:acthatLhQM}
&\quad  +\frac h{2 N_{z,h}} (G^{-1}\nabla\xi^z)\cdot\nabla\xi^z \, \chi'(\xi^z)e^{\xi^z/h}.
 \end{align}   
Note moreover that
 \eqref{eq.chi-cut}
and  \eqref{eq:estimNhz} imply that
the last term of \eqref{eq:acthatLhQM}
is of the order $O(e^{-\frac ch})$
for some $c>0$.
Hence,   in order to ensure  $\hat L_h\varphi_z=O(h^\infty)$, it is sufficient
to choose $\xi^{z}\sim\sum_{j\geq 0} h^j\xi^z_j$ such that
\be\label{eq:phase1}
\frac 12(G^{-1}\nabla\xi^z)\cdot \nabla\xi^z + b^\circ\cdot\nabla\xi^z+h\big (\rho^\circ\cdot\nabla\xi^z+ \frac 12\div(G^{-1}\nabla\xi^z)\big )  =O(h^\infty).
\ee
Identifying the powers of $h$ in \eqref{eq:phase1}, this amounts to
the following equations:
\be
\label{eq:eik}\tag{E}
G^{-1}\nabla\xi^z_0\cdot \nabla\xi^z_0+2b^\circ\cdot\nabla\xi^z_0=0
\ee
and, for all $j\geq 1$,
\be\label{eq:transp_j}\tag{T-j}
(G^{-1}\nabla\xi^z_0+b^\circ)\cdot\nabla\xi_j^z=\mathcal Q_j,
\ee
where $\mathcal Q_j$ is a function which depends smoothly on the functions $\xi^z_k$ and their derivatives for $k\in \{0,\ldots,j-1\}$. Equation  \eqref{eq:eik} is an eikonal equation while Equations \eqref{eq:transp_j}, $j\geq 1$, are transport equations.

The existence of functions $\xi_{j}^z$, $j\geq0$, satisfying these equations 
 follows from standard
results on non-linear first order PDE with non-characteristic boundary (see for example \cite[pages 7 to 9]{dimassi-sjostrand-99} or \cite[Section 3.2 in Part I]{Eva}). We are more specific
below.

\subsubsection{Resolution of the eikonal equation} 
In this section, we look for a solution $\xi^z_0\not\equiv0$ of
\eqref{eq:eik} which vanishes in a neighborhood of $0$ in 
 the hyperplane $\{v_{d}=0\}$.

The fact that $\{v_{d}=0\}$ is non-characteristic near $0$ means that 
the vector field $b^\circ$ involved in \eqref{eq:eik} is transverse to 
$\{v_{d}=0\}$ near $0$. 
Indeed, we have $b^\circ=-G^{-1}\nabla \hat f-\ell^\circ$,
$G(0)=\Id$, and $\nabla \hat f(0)= \mu_{z}e_d$  according respectively to \eqref{eq:defbtilde1}, 
\eqref{eq.G1}, and \eqref{eq.cv-pa-omega-nablafnon03}, where
 we recall that $(e_1,\ldots,e_d)$ denotes the canonical basis of $\R^d$.
 Since moreover $\nabla\hat f(0)\cdot\ell^\circ=0$ according to
\textbf{[A$\perp$]} and $\mu_{z}>0$ (see \eqref{eq:defmuHz}), it follows that
\begin{equation}
\label{eq:orthogell0}
b^\circ(0)= -\mu_{z}e_d-\ell^\circ(0)\quad\text{and}\quad e_{d}\cdot\ell^\circ(0)=0\,.
\end{equation}
Hence, the vector  $b^\circ$ is transverse to $\{v_{d}=0\}$ at  $0$, and thus near $0$ by continuity.

We can thus apply \cite[Theorem~1.5]{dimassi-sjostrand-99} to
$$p(v,\eta):=G^{-1}(v)\eta\cdot \eta+2b^\circ(v)\cdot\eta\quad\text{around $(0,\eta^*)\in\R^{d}\times\R^{d}$}$$
for any $\eta^*=(\eta'^{*},\eta_{d}^*)\in\R^{d}$ satisfying $p(0,\eta^*)=0$:
for any smooth real function $\psi$ defined near $0$ in $\R^{d-1}$ such that $\nabla_{x'}\psi(0)=\eta'^{*}$,
there exists a unique  smooth real function $\xi^z_0$ defined around $0\in \R^{d}$
such that on this neighborhood $p(x,\nabla \xi^z_0(x))=0$, $\xi^z_0(x',0)=\psi(x')$, and $\nabla \xi^z_0(0) = \eta^*$.
Since we look for $\xi^z_0\not\equiv0$ vanishing on $\{v_{d}=0\}$, this amounts to choose $\psi\equiv 0$
and thus $\eta'^{*}=0$, and
$\eta^*_{d}$ as the nonzero solution of 
$p(0,(0,\eta^*_{d}))=|\eta^*|^{2}+2b^\circ(v)\cdot \eta^*=0$,
that is to take $\psi\equiv0$, $\eta'^{*}=0$, and
$$
\eta^*_{d}:=-2b^\circ(0)\cdot e_{d}=2\mu_{z}.
$$
Note also that Taylor's theorem with integral form of the remainder then  implies that 
$\xi^z_0$ factorizes as $\xi^z_0 = v_{d} a$, where
$a$ is a smooth function defined around  $0\in \R^{d}$. 
Since 
$\nabla \xi^z_0(0)=\eta^*$, we have
in addition $a(0)=\eta^*_{d}=2\mu_{z}$, and
we have thus proved the

\begin{proposition}\label{prop:solveik}
There exists a function $a\in C^\infty (\R^d)$ 
satisfying
$a(0)=2\mu_z$ such that
 the function $\xi_0^z$ defined by $\xi_0^z(v)=v_da(v)$ 
 satisfies the eikonal equation \eqref{eq:eik} in a neighborhood of $0$.
\end{proposition}

%

\subsubsection{ Resolution of the transport equations}
 The following proposition permits to solve the equations \eqref{eq:transp_j}. 
 \begin{proposition}\label{prop:solvtranspj}
 There exists a  neighborhood    $\mathscr V$ of $0$  in $\mathbf R^d_-$ and a sequence of functions $(\xi^z_j)_{j\ge 1}$ such that for all $j\ge 1$,  $\xi^z_j\in C^\infty(\mathscr V)$ satisfies  \eqref{eq:transp_j}  on $\mathscr V$ and vanishes on $\mathscr V\cap\{v_d=0\}$.
 \end{proposition}
 \noindent
 \bp We proceed  by induction on $j\geq 1$.
 All the transport equations have the same structure, only the right hand side of \eqref{eq:transp_j} depending on the preceding step. Hence, it is sufficient to prove that there exists a neighborhood    $\mathscr V$ of $0$ such that, for any smooth function $\mathcal Q$   defined on $\mathscr V$, we can find a smooth function $u$ on $\mathscr V$  which vanishes on 
 $\{v_d=0\}$ and solves
 \be\label{eq:solvtransp0}
 (G^{-1}\nabla\xi_0^z+b^\circ)\cdot\nabla u=\mathcal Q.
 \ee
Let us  recall that $\nabla \xi^z_0(0)=2\mu_{z}e_{d}$  (see Proposition \ref{prop:solveik} and the lines above) and 
$b^\circ(0)=-\mu_z e_d-\ell^\circ_0(0)$ with $e_{d}\cdot\ell^\circ(0)=0$
(see \eqref{eq:orthogell0}). Thus, the vector field $F^\circ:=G^{-1}\nabla\xi_0^z+b^\circ$ satisfies 
$F^\circ(0)=\mu_z e_d-\ell^\circ_0$ and
is 
transverse to $\{v_{d}=0\}$ around~$0$.

Hence, the  characteristics curves 
$$
\frac {d}{dt}y_t=F^\circ(y_t), \ y_0(v')=v'$$
define a smooth coordinate system $(t,v')\in \mathbf R_-\times \mathbf R^{d-1}\mapsto y_t(v')$ near $0$ in $\mathbf R_-^d$. We then define $u$ in these coordinates by $
u(y_t(v'))=\int_0^t\mathcal Q(y_s(v'))ds$. 
It is straightforward to deduce that  the function $u$ solves \eqref{eq:solvtransp0}. Moreover taking $t=0$ in the above equation, one gets 
$u(v',0)=0$. The proof of the proposition is complete. 
\ep

A Borel construction then leads to the existence of a solution $\xi^z$ to \eqref{eq:phase1}.

\begin{proposition}\label{prop:definQMbord}
Let $(\xi^z_j)_{j\geq 0}$ be the sequence of functions given by Propositions~\ref{prop:solveik} and~\ref{prop:solvtranspj}. 
Then, there exist a  neighborhood    $\mathscr V$ of $0$  in $\mathbf R^d_-$ and a family of smooth 
functions $\xi^z=\xi^{z} (\cdot , h)$, $h\in]0,1]$, 
admitting 
the $h$-classical expansion $\xi^z\sim\sum_{j\geq 0}h^j\xi^z_j$ on $\mathscr V$,
which vanishes on $\mathscr V\cap\{v_d=0\}$
and satisfies on $\mathscr V$,
$$
\frac 12(G^{-1}\nabla\xi^z)\cdot \nabla\xi^z + b^\circ\cdot\nabla\xi^z+h\big (\rho^\circ\cdot\nabla\xi^z+ \frac 12\div(G^{-1}\nabla\xi^z)\big )  =O(h^\infty).
$$
\end{proposition}
In particular, we have indeed built up a function $\varphi_z$ in the $v$-coordinates  such that $\hat L_h\varphi_z=O(h^\infty)$   near $0$ in $\mathbf R_-^d$.  
  
  \subsection{Extra conditions on the size parameters $\delta_1,\delta_2$}
  \label{sec.Ec} 
  Recall  that for $\delta_1, \delta_2>0$,  $\delta=(\delta_1,\delta_2)$  measures the sizes of the cylinders  $\mathscr V^{\delta}_z=v (\mathscr U^{\delta}_z)$ (see \eqref{eq:defVz}). 
In this section, we  adjust   the size parameters $\delta_1,\delta_2$ to get the extra conditions  \textbf{[C$^\delta_4$]} to \textbf{[C$^\delta_8$]} below which will be needed  in the quasi-modal estimates of Section \ref{sec.LAL} and in the definition of the quasi-mode $u_h^{\rm{app}}$ (see \eqref{eq:defnormquasi-mode} below). This adjustment is made while preserving the properties  \textbf{[C$^\delta_1$]} to \textbf{[C$^\delta_3$]} of  these neighborhoods   which were  imposed in Section \ref{sec:coord}.

First of all, introduce the following notation (see \eqref{eq:defmuHz})  for $v\in \mathscr V^{\delta}_z$:
 \be\label{eq:deftildef}
Q(v):=\mu_zv_d+\frac 12 H_zv'\cdot v' \text{ and } Q^+(v)=-\mu_z v_d+\frac 12 H_zv'\cdot v'.
\ee 
Note that $Q^+\ge 0$. 
Recall also that $\mu_z>0$, $H_z$ is a positive definite matrix, and that according to Proposition \ref{prop:solveik}, 
 $$\xi_0^z= 2\mu_z v_d+ v_d \mathscr A(v),$$
  where  $\mathscr A$ is a smooth function such that $\mathscr A(0)=0$.
Therefore,  using \eqref{eq:deftildef}, we get for $v\in \mathscr V^{\delta}_z$:\be\label{eq:phaseapp}
Q(v)-\xi_0^z(v) =Q^+(v) -v_d\mathscr A(v)=- v_d (\mu_z +\mathscr A(v))+\frac 12 H_zv'\cdot v'.
\ee
Note also that $|v_d\mathscr A(v)|\le c |v_d||v|\le c|v_d|^2+ c|v_d||v'|\le c(1+ \epsilon^{-1})|v_d|^2 + c\epsilon |v'|^2$ for some $c>0$ and any $\epsilon>0$.  
Thus, by \eqref{eq:phaseapp}, one has:
$$Q(v)-\xi_0^z(v)\ge -\mu_z v_d- c(1+ \epsilon^{-1})|v_d|^2 +\frac 12 H_zv'\cdot v'   - c\epsilon |v'|^2.$$
On the other hand,  recall that the support  of $1-|\chi|^2$ is included in the set $\{t\in \mathbf R, t\le -\mu_z\delta_1/4\}$ (see \eqref{eq.chi-cut}). 
Consequently, choosing $\epsilon >0$ small enough above and up to decreasing  $\delta_1, \delta_2>0$, the two  following conditions hold:
\begin{enumerate}  
 \item []\textbf{[C$^\delta_4$]}  $\argmin_{\mathscr V^{\delta}_z}(Q-\xi_0^z)=\{0\}$. 
 
   \item []\textbf{[C$^\delta_5$]}   $(1-|\chi|^2)(\xi_0^z(v))\neq 0$ and $v\in \mathscr V^{\delta}_z$ imply that  
  $$v_d\le -\delta_1/9 \text{ and } Q(v)-\xi_0^z(v)\ge   \mu_z \delta_1/10.$$
  \end{enumerate}

In addition, recall  (see Proposition \ref{prop:definQMbord}) that the family of functions  $\xi^z(\cdot,h)  \sim\sum_{j\geq 0}h^j\xi^z_j$ is defined in a $h$-independent  neighborhood of $0$ in $\mathbf R_-^d$. Therefore, up to decreasing again $\delta_1, \delta_2>0$, there exists $h_0>0$ such that for all $h\in]0,h_0]$: 
 \begin{enumerate}  
 \item []\textbf{[C$^\delta_6$]}  The function $v=(v',v_d)\mapsto \xi^z(v,h)$ is well defined 
and satisfies Proposition~\ref{prop:definQMbord}
 on $\mathscr V^{\delta}_z$, and for very $v\in \mathscr V^{\delta}_z$,  $\xi^z((v',v_d), h)<0$ when $v_d<0$.
  \item []\textbf{[C$^\delta_7$]} For any $v=(v',v_d)\in \mathscr V^{\delta}_z$, $\xi^z((v',v_d),h)\le -\delta_1\mu_z$ when $v_d\le -\delta_1$.  
\item[]\textbf{[C$^\delta_8$]} For any $v\in \mathscr V^{\delta}_z$, one has
$$
\xi^z(v,h)\in\supp \chi'\Rightarrow v_d\le -\frac{\delta_1} 9 \text{ and }Q(v)-\xi^z(v,h)\ge    \mu_z \delta_1/10.
$$
 \end{enumerate}
Note that in order to deduce that  $\xi^z((v',v_d), h)<0$ when $v_d<0$ for $v\in \mathscr V^{\delta}_z$, we used
the fact that $\xi^z((v',v_d), h)=v_{d}(2\mu_{z}+O(v)+O(h))$.

\subsection{Definition of the quasi-mode}
\label{sec.Qmm}
We  are now in position to define  $u_h^{\rm{app}}$ near the generalized saddle point $z\in \mathscr  P_{\rm sp}$ (see \eqref{eq.PP}). Let $\xi^z$ be given by Proposition \ref{prop:definQMbord} and $\chi$ satisfying \eqref{eq.chi-cut}. 
Recall that    (see \eqref{eq:defVz} and \eqref{eq.qmPhi}): 
\begin{equation}\label{eq.qm-local11}
\forall v  \in \mathscr V^{\delta}_z, \  \  \varphi_{z} (v)=\frac1{ N_{z,h}} \int_{\xi^z(v,h)}^0\chi(t)e^{\frac t h}  dt, \ \ \text{where}\ \  N_{z,h}:=\int_{-\infty}^0\chi(t)\, e^{\frac t h}  dt. 
     \end{equation} 
Note that $\varphi_z$ also depends on $h>0$ but for ease of notation, we have decided not to indicate this dependency in its notation. 
By construction of $\xi^z(\cdot,h)$, and using  \textbf{[C$^\delta_6$]} and \textbf{[C$^\delta_7$]}, it holds for $h>0$ small enough:
      \begin{equation}\label{eq.qm-local1-property12}
  \left\{
    \begin{array}{ll}
    \varphi_{z}\in   C^\infty\big (\mathscr V^{\delta}_z,[0,1]\big ),\\\varphi_z(v',0)=0, \\
\forall (v',v_d)\in \mathscr V^{\delta}_z,\, \varphi_{z}(v',v_d)=1 \text{  when  } v_d \in [-2\delta_1,-\delta_1].
 \end{array}
\right.
\end{equation}
We now want to glue all these definitions near $z\in\mathscr  P_{\rm sp}$ into a globally defined quasi-mode $u_h^{\rm{app}}$ over $\overline \Omega$ vanishing on $\partial \Omega$. Recall the conditions \textbf{[C$^\delta_1$]} to \textbf{[C$^\delta_8$]}.  
 On the other hand,   for every
$x\in \pa   \mathscr{C}_{{\rm min}}$, one has  $\nabla f(x)\neq 0$, which implies 
that for every~$r>0$ small enough, $\{f<f(x)\}\cap   B(x,r)$ is connected  and thus included in $ \mathscr{C}_{{\rm min}}$.

These considerations imply the existence of the following subsets 
$\mathscr{C}_{-}$ and~$\mathscr{C}_{+}$ of~$\Omega$.

\begin{proposition}\label{pr.omega1}  Assume {\rm \textbf{[A]}}.  Then,  there exist two $\mathcal C^\infty$  connected open sets~$\mathscr{C}_{-}$ and~$\mathscr{C}_{+}$ of $\Omega$  satisfying the following properties:
 \begin{enumerate}
 \item []{$\mathfrak 1$.} 
  $\overline{\mathscr{C}}_{{\rm min}} \subset \mathscr{C}_{+} \cup \pa \Omega$. 
  \item[]{$\mathfrak 2$.}   $\overline{\mathscr{C}}_{+}$ is a neighborhood  in $\overline\Omega$ of each set $ \mathscr U^{\delta}_z$, for $z\in\mathscr  P_{\rm sp}$.
 \item[]{$\mathfrak 3$.}      $\overline{\mathscr{C}}_{-}\subset \mathscr{C}_{+}$ and the strip $\overline{\mathscr{C}}_{+}\setminus \mathscr{C}_{-}$ satisfies
\begin{equation}\label{stripS}
\exists c>0\,,\ f\geq f(x_{0})+c \ \text{on}\ \overline{\mathscr{C}}_{+}\setminus \mathscr{C}_{-}
\quad\text{and}\quad
 \overline{\mathscr{C}}_{+}\setminus \mathscr{C}_{-}= 
 \bigcup \limits_{z\in \mathscr  P_{\rm sp}}  \mathscr U^{\delta}_z   \  \bigcup \   \mathscr O,
\end{equation}
where the subset $\mathscr O$ of $\overline\Omega	$  is such that:
\begin{equation}\label{eq.SS}
\exists c>0\,, \  f\ge \min_{\pa \Omega}f  + c  \ \text{on}\ \mathscr O.
\end{equation}
\end{enumerate}
\end{proposition}

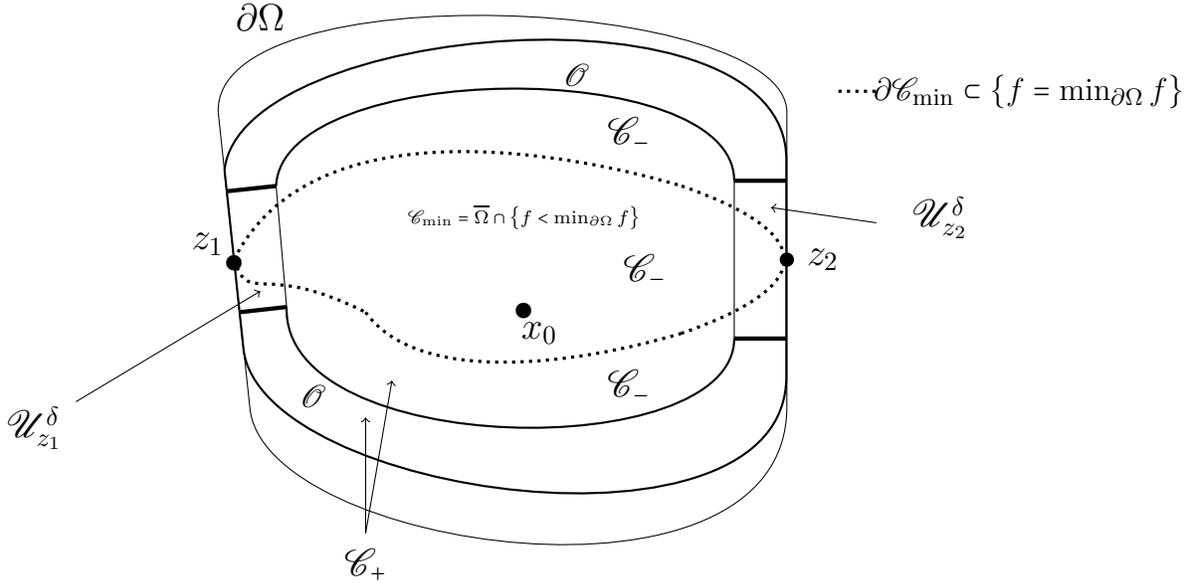
\begin{figure}[h!]
\begin{center}
\begin{tikzpicture}[scale=0.7]
\tikzstyle{vertex}=[draw,circle,fill=black,minimum size=5pt,inner sep=0pt]
\draw (-0.2,-2.8)--(-0.8,2.8);
 \draw[ultra thick]  (-0.63,1.3)--(0.3,1.4);
  \draw[ultra thick]  (-0.4,-1)--(0.5,-0.9); 
\draw  (0.3,1.4)--(0.5,-0.9);
 \draw (0,4.6) node[]{$\pa \Omega$}; 
    \draw (-0.5,-0.1)   node {\large{$\bullet$}} ; 
   \draw (5,-1)   node {\large{$\bullet$}} ;
  \draw (5.3,-1.4) node[]{$ x_0$};
    \draw (2,-5.8) node[]{$\mathscr{C}_{+}$};
    \draw [->] (2,-5.2) to (2,-3) ;
    \draw [->] (2,-5.2) to (2.5,-2.3) ;  
  \draw [dotted, very thick]   (10.98 ,3.2) -- (11.7 ,3.2);
 \draw  (14.6 ,3.2) node[]{{\small $\pa \mathscr{C}_{{\rm min}}\subset \{f= \min_{\pa \Omega}f\}$}}; 
     \draw (7,2.4) node[]{$\mathscr{C}_{-}$};
     \draw (7.3,-0.2) node[]{$\mathscr{C}_{-}$};
      \draw (7,-2.4) node[]{$\mathscr{C}_{-}$};
      \draw (5,0.8) node[]{{\tiny $\mathscr{C}_{{\rm min}}= \overline \Omega\cap\big\{ f<\min_{\pa \Omega}f \big\}$}}; 
  \draw (10,-2.8)--(10,2.8);
\draw[ultra thick]  (10,1.5)--(9,1.5);
\draw[ultra thick]  (10,-1.5)--(9,-1.5);
\draw (9,1.5)--(9,-1.5);  
  \draw (10,0)  node[vertex,label=east: {$z_2$}](v){};
  \draw[very thick, dotted]  (8,1.4) ..controls (10.6,0.5) and   (10.6,-0.5)  .. (8,-1.4) ; 
  \draw[very thick, dotted]  (-0.5,0) ..controls (-0.2,-0.9) and   (0,0)  .. (2,-1);
 \draw[very thick, dotted]  (-0.5,0) ..controls (1,3) and   (6,2)  .. (8,1.4) ;
\draw[very thick, dotted]  (2,-1) ..controls (3,-2.4) and   (6,-2)  .. (8,-1.4) ;
\draw[thick]  (0.3,1.4) ..controls (0.4,3.8) and   (9,3.9)  .. (9,1.5);
\draw[thick]  (0.5,-0.9) ..controls (0.4,-3.8) and   (9,-3.9)  .. (9,-1.5);
\draw[thick]  (-0.68,1.7) ..controls (-0.4,4.5) and   (10,5.4)  .. (9.99,1.9);
\draw[thick]  (-0.32,-1.7) ..controls (0,-4.5) and   (10,-5.9)  .. (9.99,-2.22);
\draw[thick]  (-0.32,-1.7)--(-0.68,1.7);
\draw[thick]  (9.99,-2.22)--(9.99,1.9);
\draw (1,-2.6) node[]{\small{$\mathscr O $}};
\draw (6,3.5) node[]{\small{$\mathscr O $}};
\draw (12.9,0.8) node[]{$\mathscr U^{\delta}_{z_2}$};
\draw [->] (11.7,0.7) to (9.7,0.99) ;
\draw (-4.3,-3.2) node[]{$\mathscr U^{\delta }_{z_1} $};
\draw [->] (-3.5,-2.7) to (0,-0.6) ;
 \draw (-1,0.3) node[]{${\small z_1}$};
\draw  (-0.8,2.8) ..controls (-0.9,5.5) and   (10,5)  .. (10,2.8);
\draw  (-0.2,-2.8) ..controls (0,-5.5) and   (10,-7)  .. (10,-2.8);
\end{tikzpicture}
\caption{Schematic representation of $\mathscr{C}_{-}$, $\mathscr{C}_{+}$, and $\mathscr O$ (see Proposition~\ref{pr.omega1}). On the figure, $\mathscr  P_{\rm sp}=\{z_1,z_2\}$. }
 \label{fig:S}
 \end{center}
\end{figure}

We refer to Figure~\ref{fig:S} for a schematic representation of  $\mathscr{C}_{-}$, $\mathscr{C}_{+}$, and $\mathscr O$. 
Note that  $x_0 \in \mathscr{C}_{-}$ and hence,  
\begin{equation}\label{eq.minC}
\argmin \limits_{\,  \overline{\mathscr{C}}_{+}}f=\argmin\limits_{\,  \overline{\mathscr{C}}_{-}}f=\{x_0\}.
\end{equation}

\noindent
Using the above sets $\mathscr{C}_{ + }$ and $ \mathscr{C}_{-}$, we define  a
function 
\bes w_h^{\rm{app}}:\overline \Omega \to [0,1]
\ees  as follows: 
\begin{enumerate}
\item[]{\rm\textbf{[-]}} For every $z\in \mathscr  P_{\rm sp}$, $w_h^{\rm{app}}$  is defined on the cylinder $\mathscr U^{\delta}_z$ 
(see \eqref{eq.vois-11-pc})
by
 \begin{equation}\label{eq.qm-local1x}
\forall x\in   \mathscr U^{\delta}_z, \  w_h^{\rm{app}}(x):= \varphi_{z} (v(x)),  \  
\text{ see~\eqref{eq.qm-local11}}.
  \end{equation}
\item[]{\rm\textbf{[-]}} From~\eqref{eq.qm-local1-property12}, \eqref{stripS},  and  the fact that $ \overline{\mathscr{C}}_{ - }\subset  \mathscr{C}_{ + }$  (see Proposition~\ref{pr.omega1}),    
the above function $w_h^{\rm{app}}$ satisfying \eqref{eq.qm-local1x}   can be extended to $\overline \Omega$  so  that
\begin{equation}\label{eq.psix=10}
w_h^{\rm{app}}=0 \text{ on } \overline \Omega\setminus\mathscr{C}_{ + }, \ \ \ w_h^{\rm{app}}=1 \text{ on }  {\mathscr{C}_{ - }}, \text{ and } w_h^{\rm{app}}\in C^\infty(\overline \Omega,[0,1]). 
\end{equation}
Note that $w_h^{\rm{app}}=0$ on $\partial \Omega$. 
Moreover, in view of \eqref{eq.qm-local11} and~\eqref{stripS},
 $w_h^{\rm{app}}$ can be chosen on~$\mathscr O$ such that, for some $C>0$ and   for  every $h$ small enough,  
\begin{equation} 
\label{eq.nablapsix=0}
\forall \alpha \in \mathbb N^d, \, \vert \alpha \vert \in \{1,2\}, \,    \Vert \pa^\alpha  w_h^{\rm{app}} \Vert_{L^\infty(\mathscr O )} \le Ch^{-2}. 
\end{equation}

\end{enumerate}
Notice that~\eqref{eq.psix=10} implies
\begin{equation}\label{eq.psix=nabla}
\supp \nabla  w_h^{\rm{app}} \subset \overline{\mathscr{C}}_{ + }\setminus \mathscr{C}_{ - }.
\end{equation}
 
We finally define the normalized quasi-mode $ u_h^{\rm{app}}$ over $\overline \Omega$ by
\be\label{eq:defnormquasi-mode}
 u_h^{\rm{app}}=   w_h^{\rm{app}}/ Z^{\rm{app}}_h \  \text{ where }   Z^{\rm{app}}_h=\Big(\int_\Omega  |w_h^{\rm{app}}|^2p_h(x)dx\Big)^{\frac 12}.
 \ee

\section{Proof of Theorem \ref{th:sharpeigenval}}

In all this section, we assume \textbf{[A]}.

\subsection{Action of $L_h$ on the quasi-mode}
\label{sec.LAL}
 
 \begin{proposition}\label{prop:calculQM}  Assume {\rm \textbf{[A]}}. There exist $h_0>0$ and a family of real numbers $(\zeta(h))_{h\in]0,h_0]}$ admitting a $h$-classical expansion $\zeta(h)\sim\sum_{j\geq 0}h^j\zeta_j$ with $\zeta_0$ given by \eqref{eq:defzeta0} such that, defining  
 \bes
  \lambda_h^{\rm{app}}=h^{-\frac 12}\zeta(h)e^{-\frac 2h(\min_{\partial\Omega}f-f(x_0))},
 \ees the following holds true:
\begin{enumerate}
\item[]{\rm [$\mathfrak 1$]} $\langle L_hu_h^{\rm{app}},u_h^{\rm{app}} \rangle_{L^2(p_hdx)}  =     \lambda_h^{\rm{app}}  (1+O(h^\infty))$
\item[]{\rm [$\mathfrak 2$]} $\Vert L_h  u_h^{\rm{app}}  \Vert^2_{L^2(p_hdx)}= O(h^\infty ) \,   \lambda_h^{\rm{app}}$.   
\item[]{\rm [$\mathfrak 3$]} $\Vert L_h^\dag  u_h^{\rm{app}}  \Vert^2_{L^2(p_hdx)}= O( h^{-1} \lambda_h^{\rm{app}})$, where $L_h^\dag$ is the adjoint of the operator $L_h$ in $L^2(\Omega, p_hdx)$. 
\end{enumerate} 
 \end{proposition}
 
 \begin{proof}
 In what follows, $C,c>0$ are  constants independent of $h>0$ and of $x\in \overline \Omega$ which can change from one occurence to another.
 We start by estimating $Z^{\rm{app}}_h$ in the limit $h\to 0$. 
 \medskip
 
 \noindent
  \textbf{Asymptotic equivalent of $Z^{\rm{app}}_h$}. 
Recall that $u_h^{\rm{app}}= w_h^{\rm{app}}/Z^{\rm{app}}_h$, see \eqref{eq:defnormquasi-mode}.
From   \eqref{eq.psix=10}, we have
\begin{align*}|Z^{\rm{app}}_h|^2&=\int_\Omega |w_h^{\rm{app}}(x)|^2 p_h(x)dx
=\int_{\mathscr{C}_{ - }}|w_h^{\rm{app}}(x)|^2 p_h(x)dx+\int_{\mathscr{C}_{ + }\setminus\mathscr{C}_{ - }}|w_h^{\rm{app}}(x)|^2 p_h(x)dx.
\end{align*}
Recall that from Theorem \ref{th.sheu},  $p_h(x)=h^{-\frac d2}R^h(x)e^{-\frac 2h (f(x)-f(x_0))}$, where
$R^h$ admits a $h$-classical expansion 
$R^h\sim\sum_{j\geq 0} h^jR_j$ on $\R^{d}$.
In particular $R^h$ is uniformly bounded on the compact 
$\overline\Omega\subset\R^{d}$. Since  $f\ge f(x_0)+c$ on $\mathscr{C}_{ + }\setminus\mathscr{C}_{ - }$ (see~\eqref{stripS})  and $|w_h^{\rm{app}}|\le 1$, it follows that
\bes
\int_{\mathscr{C}_{ + }\setminus\mathscr{C}_{ - }}|w_h^{\rm{app}}(x)|^2 p_h(x)dx=O(e^{-\frac {c}h }).
\ees
 Moreover,  using the fact that  $w_h^{\rm{app}} = 1$ on $\mathscr{C}_{ - }$, and Lemma 1.2 of \cite{sheu1986}, we get 
 \bes
 \int_{\mathscr{C}_{ - }}|w_h^{\rm{app}}(x)|^2 p_h(x)dx=\int_{\mathscr{C}_{ - }} p_h(x)dx=\int_{\R^d}p_h(x)dx+O(e^{-\frac {c}h })
 =1+O(e^{-\frac {c}h }).
 \ees
We thus obtain that in the limit $h\to 0$, 
 \be\label{eq:estimAah}
Z^{\rm{app}}_h=1+O(e^{-\frac ch}).
 \ee

 \noindent
 \textbf{Proof of} {\rm [$\mathfrak 1$]}.  
 Note that $w_h^{\rm{app}}\in \mathscr D_i$, for all $i=1,2,3$. Then, from \eqref{eq.quadra}, and the fact that $w_h^{\rm{app}}$ is supported in $\overline{\mathscr{C}}_{+}$ and equal to $1$ on $\mathscr{C}_{ - }$, we have   \bes
 \langle L_h w_h^{\rm{app}},w_h^{\rm{app}}\rangle_{L^2(p_hdx)}=\frac h2 \int_\Omega |\nabla w_h^{\rm{app}}(x)|^2 \, p_h(x)\, dx
 =\frac h2 \int_{\mathscr{C}_{ + }\setminus \mathscr{C}_{ - }} |\nabla w_h^{\rm{app}}(x)|^2 \, p_h(x)\, dx.
\ees   
   Using \eqref{stripS}, this implies that 
   \bes
   \langle L_hw_h^{\rm{app}},w_h^{\rm{app}}\rangle_{L^2(p_hdx)}=\frac h2\int_ {\mathscr O}|\nabla w_h^{\rm{app}}(x)|^2 \, p_h(x)\, dx+\frac h2\sum_{z\in \mathscr  P_{\rm sp}} \int_{ \mathscr U^{\delta}_z}|\nabla w_h^{\rm{app}}(x)|^2 \, p_h(x)\, dx.
   \ees
Moreover, thanks to Theorem \ref{th.sheu} and \eqref{eq.SS},  for all $x\in \mathscr O$,  $p_h(x)\leq C e^{-2(\min_{\partial\Omega}f-f(x_0)+c)/h}$ and thanks to \eqref{eq.nablapsix=0}, one deduces that
      \be\label{eq:prodscal1}
   \langle L_hw_h^{\rm{app}},w_h^{\rm{app}}\rangle_{L^2(p_hdx)}=\frac h2\sum_{z\in \mathscr  P_{\rm sp}}  \boldsymbol \vartheta_{z,h}+O(e^{-\frac 2h ( \min_{\partial\Omega}f-f(x_0)+c))}),
   \ee
 where
$$
  \boldsymbol \vartheta_{z,h}:= \int_{ \mathscr U^{\delta}_z}|\nabla w_h^{\rm{app}}(x)|^2 \, p_h(x)\, dx.
$$
   We now estimate each integral $  \boldsymbol \vartheta_{z,h}$ when $h\to 0$. To this end, fix  $z$ in $\mathscr  P_{\rm sp}$. Recall the coordinates $x\mapsto v(x)$ defined over $\mathscr U_z$ in Section~\ref{sec:coord}, 
see \eqref{eq.cv-pa-omega-nablafnon0}--\eqref{eq.G1}. We also recall that  $\hat f=f\circ v^{-1}$ (see \eqref{eq.cv-pa-omega-nablafnon03}), $  {\hat \ell} ={\ell}\circ v^{-1}$, and that for any smooth function $u$, we have $ (\nabla u )(v^{-1}) =  {}^tJ^{-1}    \nabla\hat  u$.  Using the $v$-coordinates, together with \eqref{eq:defVz} and \eqref{eq.qm-local1x},  one has  
$$
  \boldsymbol \vartheta_{z,h}  =
 \int_{ \mathscr V^{\delta}_z}(G^{-1}\nabla \varphi_z)\cdot \nabla \varphi_z \  
\hat p_h\,\sqrt{|G|}(v) dv.
$$
 Let $\xi^z$ be given by Proposition \ref{prop:definQMbord}. 
Recall that by \eqref{eq.qm-local11}, one has for all $v\in \mathscr V^\delta_z$, 
\bes
\nabla \varphi_z(v)=-\frac1{N_{z,h}}\chi(\xi^z(v,h))e^{\xi^z(v,h)/h}\nabla\xi^z(v,h).
\ees 
Using \eqref{eq.Rhh}, we then obtain that:
\begin{align*}
  \boldsymbol \vartheta_{z,h}&=\frac{h^{-\frac d2}}{|N_{z,h}|^2} \int_{ \mathscr V^{\delta}_z}G^{-1}\nabla \xi^z(v,h)\cdot \nabla \xi^z(v,h)|\chi( \xi^z(v,h))|^2\\\
&\quad \times  \hat R^h (v)
e^{-\frac 2h (\hat f(v)-f(x_0)-\xi^z(v,h))}
\sqrt{|G|}(v) dv.
\end{align*}
 Moreover, using \eqref{eq.cv-pa-omega-nablafnon03} and the fact that $f(z)=\min_{\partial\Omega}f$, one has 
$\hat f(v)-f(x_0)=\min_{\partial\Omega}f-f(x_0)+Q(v)$, see \eqref{eq:deftildef}.  
Then, one has
\begin{align*}
&{\textstyle \frac{|N_{z,h}|^2}{h^{-\frac d2}}\, e^{\frac 2h(\min_{\partial\Omega}f-f(x_0))} }  \boldsymbol \vartheta_{z,h}\\
&=   \int_{ \mathscr V^{\delta}_z}{\textstyle G^{-1}\nabla \xi^z(v,h)\cdot \nabla \xi^z(v,h)|\chi( \xi^z(v,h))|^2\, \hat R^h (v)
e^{-\frac 2h (Q(v)-\xi^z(v,h))}}
\sqrt{|G|}(v) dv.
\end{align*}
Recall that 
$\xi^z(v,h)\sim\sum_{j\geq 0}h^j\xi^z_j$ and $R^h(v)\sim\sum_{j\geq 0}h^jR_j(v)$, see respectively Proposition \ref{prop:definQMbord} and Theorem \ref{th.sheu}. Recall also that
according to Propositions~\ref{prop:solveik} and~\ref{prop:solvtranspj},
 $\xi_0^z$ solves the eikonal equation \eqref{eq:eik}  and $\xi^z_1$  solves the first transport equation \eqref{eq:transp_j}. 
We deduce that  for any $N\in \N$,  
\be\label{eq:expandclassI}
\frac{|N_{z,h}|^2}{h^{-\frac d2}}\,   e^{\frac 2h(\min_{\partial\Omega}f-f(x_0))} \,   \boldsymbol \vartheta_{z,h}= \sum_{j=0}^N h^j   \boldsymbol \vartheta_{z,h}^{(j)} \ +\ O(h^{N+1}),
\ee 
where  $(  \boldsymbol \vartheta_{z,h}^{(j)})_{j\ge 0}\subset \mathbf R$ and 
\bes
  \boldsymbol \vartheta_{z,h}^{(0)}= \int_{ \mathscr V^{\delta}_z}G^{-1}\nabla \xi_0^z(v)\cdot \nabla \xi_0^z(v)\hat R_0(v)|\chi(\xi_0^z(v))|^2
e^{2\xi^z_1(v)}
e^{-\frac 2h (Q(v)-\xi_0^z(v))}
\sqrt{|G|}(v)dv.
\ees
Note that we have used \textbf{[C$^\delta_4$]} in see Section \ref{sec.Ec}, and more precisely that $Q(v)-\xi_0^z(v)\ge 0$, 
to obtain the remainder term  $O(h^{N+1})$ in \eqref{eq:expandclassI}.

Recall \eqref{eq:phaseapp}. Using the condition \textbf{[C$^\delta_5$]} (see Section \ref{sec.Ec}), we deduce that 
\begin{align}
\nonumber
  \boldsymbol \vartheta_{z,h}^{(0)}&=  \int_{ \mathscr V^{\delta}_z}G^{-1}\nabla \xi_0^z(v)\cdot \nabla \xi_0^z(v)\ \hat R_0(v)\ e^{2\xi^z_1(v)}
e^{-\frac 2h  ( Q^+(v)-v_d \mathscr A(v))}
\sqrt{|G|}(v)dv \\
\label{eq:formI0}
 &\qquad +O(e^{-\frac ch}).
\end{align}
By Propositions \ref{prop:solveik} and \ref{prop:solvtranspj}, one has  $|\nabla \xi_0^z(0)|^2=4 |\mu_z|^2$  and $\xi^z_1(0)=0$. 
Then, performing a Taylor expansion, we deduce  from \eqref{eq.R0} and \eqref{eq.G1}    that there exists a  sequence  $(\boldsymbol \theta_\alpha)_{\alpha\in\N^d}\subset \mathbf R$ such that for any $K\in\N$:
\be\label{eq:taylorI0}
\sqrt{|G|}(v)\, G^{-1}(v)\nabla \xi_0^z(v)\cdot \nabla \xi_0^z(v)\hat R_0(v)
\, e^{2\xi^z_1(v)}=\sum_{|\alpha|\leq K}\boldsymbol \theta_\alpha v^\alpha+O(|v|^{K+1}),
\ee
where 
\be\label{eq:computtheta0}
\boldsymbol \theta_0=4|\mu_z|^2c_0\exp\Big[\int_ 0^{+\infty} \div (\ell )(\psi_t(z))dt\Big]
\ee
and  $c_0>0$ is given by \eqref{eq:calc0}. 
For $\alpha \in \N^d$, we write $\alpha=(\alpha',\alpha_d)$ with $\alpha'\in \N^{d-1}$ and $\alpha_d\in \N$. Set $v_h=(\sqrt h v',hv_d)$.  
Then, combining 
\eqref{eq:formI0} and   \eqref{eq:taylorI0}, it follows that for any $K\in \N$, using a change of variables, 
\bes
\begin{split}
  \boldsymbol \vartheta_{z,h}^{(0)}&= \sum_{|\alpha|\leq 2K}\boldsymbol \theta_\alpha  \int_{ \mathscr V^{\delta}_z} v^\alpha e^{-\frac 2h ( Q^+(v)-v_d \mathscr A(v))} dv + \int_{ \mathscr V^{\delta}_z} O(|v|^{2K+1})e^{-\frac 2h ( Q^+(v)-v_d \mathscr A(v))}  dv\\
&= h^{\frac{d+1}{2}} \Big[\sum_{\frac {|\alpha'|} 2+\alpha_d\leq K} \boldsymbol \theta_\alpha h^{\frac {|\alpha'|} 2+\alpha_d} \int_{\vert v'\vert \le \delta_2/\sqrt h}\int_{-2\delta_1/h}^0 v^\alpha e^{-2(Q^+(v)-v_d\mathscr A(v_h))} dv +O(h^{K+\frac12}) \Big].
\end{split}
\ees  
Recall that $\mathscr A(0)=0$. 
  Hence,  performing a second Taylor expansion
 and denoting, for $k\in\N$, by $(\mathbf v_{\beta})_{\beta\in\N^{d},|\beta|=k}$ the monomial basis of the homogeneous polynomials of order $k$, 
 there exists a family $(\lambda_{\beta})_{\beta\in\N^{d}}\subset \mathbf R$ such that,
 for every $K\in\N$,
\begin{align*}
2v_d\mathscr A(v_h)&= \sum_{|\beta|\ge 1, \, \frac {|\beta'|} 2+\beta_d \le K} h^{\frac {|\beta'|} 2+\beta_d} \, 2\lambda_\beta v_d \mathbf v_{\beta}(v)  +  O(h^{K+ \frac{1}2} ).
\end{align*}
Thus,  it holds
 $$e^{2v_d\mathscr A(v_h)}= 1+ \sum_{|\gamma|\ge 1,    {|\gamma'|} /2+ \gamma_d \le K}     h^{\frac {|\gamma'|} 2+ \gamma_d}  \mathscr A_\gamma(v) +O(h^{K+\frac 12}),$$
  where each $\mathscr A_\gamma(v)$ is a linear combination of   monomials of the form $ v_d^m   \mathbf v_{\beta_1} (v)$ $\ldots \mathbf v_{\beta_m}(v)$ with $m\ge 1$ and   where  
  $$\gamma=(\beta_1'+\ldots+\beta_m', (\beta_1)_d+\ldots+(\beta_m)_d)\in  \mathbf N^d\ \ \text{and}\ \ 
  \frac 12{|\gamma'|}  + \gamma_d = \frac 12 \sum_{i=1}^m |\beta_i'| + \sum_{i=1}^m  (\beta_i)_d \ge \frac12.$$
Therefore, there exist   coefficients $  \boldsymbol \theta^*_\alpha$ such that  
 $  \boldsymbol \theta^*_0=   \boldsymbol \theta_0$ and
\bes
\begin{split}
  \boldsymbol \vartheta_{z,h}^{(0)}&=h^{\frac{d+1}{2}} \Big[\sum_{\frac {|\alpha'|} 2+\alpha_d\leq K}  \boldsymbol \theta^*_\alpha h^{\frac {|\alpha'|} 2+\alpha_d} \int_{\vert v'\vert \le \delta_2/\sqrt h}\int_{-2\delta_1/h}^0v^\alpha P_\alpha(v_d)e^{-2Q^+(v)} dv+O(h^{K+\frac12})\Big]\\
&=h^{\frac{d+1}{2}} \Big[\sum_{\frac {|\alpha'|} 2+\alpha_d\leq K} \boldsymbol \theta^*_\alpha h^{\frac {|\alpha'|} 2+\alpha_d} \int_{v'\in\R^{d-1}}\int_{v_d<0} v^{ \alpha} P_\alpha(v_d)e^{-2Q^+(v)} dv+O(h^{K+\frac12})\Big],
\end{split}
\ees
where for $\alpha \in \mathbf N^d$,  $P_\alpha\in \mathbf R[X]$. 
  For a parity reason,  if $\alpha'$ is an odd number, $\int_{v'\in\R^{d-1}} v^{\alpha} e^{-2Q^+(v)} dv'=0$. Hence, we deduce that as $h\to 0$,
 $$  \boldsymbol \vartheta_{z,h}^{(0)}\sim h^{\frac{d+1}{2}}\sum_{k\geq 0}h^k K^{(0)}_{z,k}$$ 
 for some sequence  $(K^{(0)}_{z,k})_{k\ge 0}\subset \mathbf R$
such that
\bes
K^{(0)}_{z,0}=\boldsymbol \theta_0\int_{\R^{d-1}\times \R^-}e^{2\mu_z v_d-H_zv'\cdot v'}dv=\frac{\boldsymbol \theta_0\pi^{\frac {d-1} 2}}{2\mu_z\sqrt{\det H_z}}.
\ees
Combined with  \eqref{eq:computtheta0} and \eqref{eq:calc0}, this gives
$$
K^{(0)}_{z,0}=\frac{2\mu_z}{\sqrt \pi}\frac {\sqrt{\det\Hess(f)(x_0)}}{\sqrt{\det H_z}} \exp\Big[\int_ 0^{+\infty} \div (\ell )(\psi_t(z))dt\Big].
$$
Similar arguments show that for any $j\geq1$, there exists a sequence  $(K^{(j)}_{z,k})_{k\ge 0}\subset \mathbf R$ such that
$$  \boldsymbol \vartheta_{z,h}^{(j)}\sim h^{\frac{d+1}{2}}\sum_{k\geq 0}h^k K^{(j)}_{z,k}.$$
Using \eqref{eq:prodscal1}, \eqref{eq:expandclassI}, and \eqref{eq:estimNhz}, we finally obtain
\bes
\langle L_hw_h^{\rm{app}},w_h^{\rm{app}}\rangle_{L^2(p_hdx)}\sim \frac{e^{-\frac 2h(\min_{\partial\Omega}f-f(x_0))}}{\sqrt h}\sum_{j\geq 0}h^j\zeta_j(z) \text{ \ \ with \ \ $\zeta_0(z)=\sum_{z\in \mathscr  P_{\rm sp}}\frac 12 K^{(0)}_{z,0}$.}
\ees
  Moreover since $u_h^{\rm{app}}=\frac {w_h^{\rm{app}}}{Z^{\rm{app}}_h}$ and $Z^{\rm{app}}_h=1+O(e^{-\frac ch})$ (see \eqref{eq:estimAah}), we deduce that 
\bes
\langle L_hu_h^{\rm{app}},u_h^{\rm{app}}\rangle_{L^2(p_hdx)}\sim \frac{e^{-\frac 2h(\min_{\partial\Omega}f-f(x_0))}}{\sqrt h}\sum_{j\geq 0}h^j\zeta_j(z). 
\ees
This  completes the proof of {\rm [$\mathfrak 1$]}.  
 \medskip

 \noindent
 \textbf{Proof of} {\rm [$\mathfrak 2$]}.  Since $w_h^{\rm{app}}$ is supported in $\overline{\mathscr{C}}_+$, equal to $1$ in $\mathscr{C}_-$ and since $L_h$ has no zero order term, then 
 \bes
 \Vert L_hw_h^{\rm{app}}\Vert_{L^2(p_h dx)}^2=\int_{\mathscr{C}_+\setminus \mathscr{C}_-} |L_hw_h^{\rm{app}}(x)|^2 p_h(x)dx.
 \ees
Since $\overline{\mathscr{C}}_{+}\setminus \mathscr{C}_{-}= 
 \bigcup \limits_{z\in\mathscr  P_{\rm sp}}  \mathscr U^{\delta}_z   \  \bigcup \   \mathscr O$, thanks to \eqref{eq.nablapsix=0}, \eqref{eq.SS} and Theorem \ref{th.sheu}, we get
 \be\label{eq:computLhwa1}
 \Vert L_hw_h^{\rm{app}}\Vert_{L^2(p_h dx)}^2=\sum_{z\in\mathscr  P_{\rm sp}}
 \int_{\mathscr U^{\delta}_z} |L_hw_h^{\rm{app}}(x)|^2 p_h(x)dx+O(e^{-\frac 2h(\min_{\partial\Omega}f-f(x_0)+c)}).
 \ee
 Using the $v$-coordinates, it follows from \eqref{eq.qm-local1x} that
 \be\label{eq:defJz}
J_{z,h}:= \int_{\mathscr U^{\delta}_z} |L_hw_h^{\rm{app}}(x)|^2 p_h(x)dx=\int_{\mathscr V^{\delta}_z}|\hat L_h\varphi_z(v)|^2\hat p_h(v)\sqrt{|G|}(v) dv.
 \ee
 On the other hand,  recall \eqref{eq:acthatLhQM}:
 \begin{align*}
 \hat L_h \varphi_z&=\frac {\chi(\xi^z)e^{\xi^z/h}}{N_{z,h}}\Big[\frac 12(G^{-1}\nabla\xi^z)\cdot \nabla\xi^z +\big[b^\circ\cdot\nabla\xi^z+h\big (\rho^\circ \cdot\nabla\xi^z+ \frac 12\div(G^{-1}\nabla\xi^z)\big ) \big ]  \Big]\\
&\quad  +\frac h{2 N_{z,h}} (G^{-1}\nabla\xi^z)\cdot\nabla\xi^z \, \chi'(\xi^z)e^{\xi^z/h}.
 \end{align*}
 Recall also that $\hat f=f(z)+Q$ (see \eqref{eq.cv-pa-omega-nablafnon03}).
Then, using Theorem~\ref{th.sheu} and  Proposition~\ref{prop:definQMbord}, one deduces that (see \eqref{eq:deftildef})
\begin{align*}
 J_{z,h}&=\frac {h^{-\frac d 2}e^{-\frac 2h(\min\limits_{\partial\Omega}f-f(x_0))}}{|N_{z,h}|^2}
 \int_{\mathscr V^{\delta}_z}\big|O(h^\infty )+ O(h) \chi'(\xi^z)\big|^{2}\hat R^h (v) e^{-\frac 2h (Q(v)-\xi^z(v,h))}dv
\end{align*} 
 Moreover, thanks to \textbf{[C$^\delta_8$]} in Section \ref{sec.Ec},  there exists $c>0$ such that 
 $Q(v)-\xi^z(v,h)\geq c$ for every  $v\in \mathscr V^{\delta}_z$ satisfying $\xi^z(v,h)\in\supp(\chi')$. Hence, we have
\begin{align*}
 J_{z,h}&=\frac {h^{-\frac d 2}e^{-\frac 2h(\min\limits_{\partial\Omega}f-f(x_0))}}{|N_{z,h}|^2}\int_{\mathscr V^{\delta}_z}O(h^\infty)\hat R^h (v) e^{-\frac 2h (Q(v)-\xi^z(v,h)) }dv
 +O(h^\infty e^{-\frac 2h(\min\limits_{\partial\Omega}f-f(x_0))}))\\
 &=O(h^\infty e^{-\frac 2h(\min\limits_{\partial\Omega}f-f(x_0))})).
\end{align*}
 In view of  
 \eqref{eq:computLhwa1} and \eqref{eq:defJz}, this yields the desired result, namely {\rm [$\mathfrak 2$]}.
 \medskip

 \noindent
 \textbf{Proof of} {\rm [$\mathfrak 3$]}.  We start with the computation of the adjoint $L_h^\dag $ of $L_h$ in $L^2(p_hdx)$. Set  
 $$\psi_h=-h\ln(p_h).$$
 Recall that the function $\psi_h$ is well-defined and smooth. Note that 
 $\nabla \psi_h= (-h\nabla R^h+2 R^h \nabla f)/ R^h$.
Hence, according to Theorem \ref{th.sheu} and formula \eqref{eq.R0},    for  every $h>0$ small enough and   $\alpha\in\N^d$,  $\partial^\alpha \psi_h$ is bounded uniformly with respect to $h$ on $\overline\Omega$. On the other hand, since $L_h^\star \, p_h=0$ (see \eqref{eq.Lstar}), a   direct computation shows that
\be\label{eq:formLhdag}
L_h^\dag  
= -\frac h2 \Delta  + (b+\nabla \psi_h) \cdot \nabla. 
\ee
Now, the same computation as the one leading to \eqref{eq:computLhwa1}
shows that
 \be\label{eq:computLhdagwa1}
 \Vert L_h^\dag w_h^{\rm{app}}\Vert_{L^2(p_h dx)}^2=\sum_{z\in\mathscr  P_{\rm sp}}
 \int_{\mathscr U^{\delta}_z} |L_h^\dag w_h^{\rm{app}}(x)|^2 p_h(x)dx+O(e^{-\frac 2h(\min_{\partial\Omega}f-f(x_0)+c)}).
 \ee
 Using the $v$-coordinates, we have: 
  \bes
 \Vert L_h^\dag w_h^{\rm{app}}\Vert_{L^2(p_h dx)}^2=\sum_{z\in\mathscr  P_{\rm sp}}
 \int_{\mathscr V^{\delta}_z} |\hat L_h^\dag \varphi_z(v)|^2 \hat p_h(v)\sqrt {|G|}dv+O(e^{-\frac 2h(\min_{\partial\Omega}f-f(x_0)+c)}).
 \ees
 Note that    
$\hat L_h^\dag =-\frac h 2\div\circ\, G^{-1}\circ\nabla+e_h\cdot\nabla$, where   $e_h$ is a smooth vector field  uniformly bounded with respect to $h$.   In particular, one has for every  $h$ small enough
 $$\hat L_h^\dag \varphi_z=O(  {e^{\xi^z/h}}/{N_{z,h}}).$$
  Using \eqref{eq:estimNhz} and the fact that 
 $$\hat p_h e^{2\xi^z/h}=h^{-\frac d2} e^{-\frac 2h(\min_{\partial\Omega}f-f(x_0))}\hat R^h e^{-\frac 2h (Q-\xi^z)},$$ we get,
since 
 $\int_{ \mathscr V^{\delta}_z}e^{-\frac 2h (Q-\xi^z)}dv=O\big(\int_{ \mathscr V^{\delta}_z}e^{-\frac 2h (Q-\xi_{0}^z)}dv\big)=O(h^{\frac{d+1}{2}})$,
 \bes
 \Vert L_h^\dag w_h^{\rm{app}}\Vert_{L^2(p_h dx)}^2=O(h^{-\frac 32} e^{-\frac 2h(\min_{\partial\Omega}f-f(x_0))})=O(h^{-1}) \lambda_h^{\rm{app}}.
 \ees
 This concludes the proof of {\rm [$\mathfrak 3$]}.
 \end{proof}
 
%
 \subsection{Proof of Theorem \ref{th:sharpeigenval}}
According to Theorem \ref{th.CMP}, for all $\beta>0$ small enough, the projector 
$$
 \pi_h:= \frac{1}{2i\pi}\int_{\{\vert z\vert =\beta \}} ( z -L_h)^{-1} dz
 $$
 is of rank one for all $h>0$ small enough, and more precisely $\Ran \pi_h= \Span (u_h)$. Moreover, Theorem \ref{th.CMP} implies that there exists $C>0$ such that for all $|z|=\beta,\;\Vert  ( z -L_h)^{-1}\Vert_{L^2(p_h dx)}\leq C$. In particular, one has $\pi_h=O(1)$ and for all 
 $u\in \mathscr D_2$, 
 $ \Vert  (1-\pi_h)  u \Vert_{L^2(p_hdx)}\le  C  \Vert L_hu   \Vert_{L^2(p_hdx)}$, which follows from the identity
 $$(1-\pi_h)u =\frac{-1}{2\pi i}\int_{\{\vert  z\vert =\beta \}}  z^{-1}\big    ( z-L_h)^{-1} L_hu \, d  z.$$
Thus, thanks  to Item  {\rm [$\mathfrak 2$]} in Proposition \ref{prop:calculQM}, one obtains that in the limit $h\to 0$:
\be\label{eq:estimprojQM}
 \pi_h   u_h^{\rm{app}} =  u_h^{\rm{app}}+  O(h^\infty | \lambda_h^{\rm{app}}|^{\frac 12})\ \ \text{in $L^2(\Omega,p_hdx)$}.
 \ee 
We thus have for every $h$ small enough,
\begin{align*}
 \lambda_h =\frac{\langle L_h  \pi_h u_h^{\rm{app}}, \pi_h u_h^{\rm{app}} \rangle_{L^2(p_hdx)}}{\|\pi_h u_h^{\rm{app}}\|_{L^2(p_hdx)}^{2}} &=\frac{ \langle L_hu_h^{\rm{app}} ,u_h^{\rm{app}} \rangle_{L^2(p_hdx)} +   E_h}{1+O(e^{-\frac ch})}\\
 &=   \lambda_h^{\rm{app}}  (1+O(h^\infty))+  E_h(1+O(e^{-\frac ch})),
\end{align*}
 where $E_h:=\langle L_h ( \pi_h-1)   u_h^{\rm{app}}, u_h^{\rm{app}}\rangle_{L^2(p_hdx)} +  \langle L_h  \pi_h  u_h^{\rm{app}},( \pi_h-1) u_h^{\rm{app}}\rangle_{L^2(p_hdx)}$  is the so-called projection error. It satisfies
 \bes
 | E_h|\le C   \Vert (\pi_h-1)  u_h^{\rm{app}}  \Vert_{L^2(p_hdx)}(\Vert L_h^\dag  u_h^{\rm{app}}  \Vert_{L^2(p_hdx)}+\Vert L_h  u_h^{\rm{app}}  \Vert_{L^2(p_hdx)})\ees
 which, combined with \eqref{eq:estimprojQM} and  {\rm [$\mathfrak 3$]} of Proposition \ref{prop:calculQM}, yields
\begin{align*}
 | E_h| =O(h^\infty)| \lambda_h^{\rm{app}}|^{\frac 12} \times O(h^{-\frac12})| \lambda_h^{\rm{app}}|^{\frac 12} + O(h^\infty) \lambda_h^{\rm{app}}=O(h^\infty) \lambda_h^{\rm{app}}. 
 \end{align*}
   This proves the desired result using Item  {\rm [$\mathfrak 1$]} in Proposition \ref{prop:calculQM}.







%
 \subsection{On  the  assumption {\rm \textbf{[A$\perp$]}}}
 \label{sec.Ap} 
  In statistical physics, we are typically given a vector field $b:\mathbf R^d\to \mathbf R^d$, that we assume to be smooth.  Metastability occurs when there exist several  stable equilibrium positions for  the dynamics \eqref{eq.flow}, see e.g.~\cite{di-gesu-lelievre-le-peutrec-nectoux-17} and references therein. Let $x_0$ be one of them and assume that
  \begin{equation}\label{eq.Ei}
  \text{all the eigenvalues of $\Jac b(x_0)$ have negative real parts.}
  \end{equation}
  Let us denote by  $B_{x_0}$  the basin of attraction of $x_0$ for the dynamics \eqref{eq.flow} 
and by $V(x)=V(x_0,x)$ the Freidlin-Wentzell quasi-potential from $x_0$ to $x\in \overline  B_{x_0}$ of the process \eqref{eq.langevin} in  $\overline B_{x_0}$ (see e.g.~\cite{FrWe} or~\cite[Equation (2.14)]{day-darden-85}).   
According to~\cite[Corollaries 2 and  7]{day-darden-85}, 
   $V(x)>0$ over $\mathbf R^d\setminus \{0\}$ and 
   there exists a connected open and dense subset $G_{x_0}$ of $B_{x_0}$ containing $x_0$ such that   $x \in G_{x_0}\mapsto V(x)$ is $\mathcal C^\infty$.   By~\cite[Section 4.3]{FrWe}, one deduces that  there exists a smooth vector field $\ell : G_{x_0}\to \mathbf R^d$ such that $b$ admits the following transversal decomposition over $G_{x_0}$:
   $$b= - \frac 12\nabla V + \ell  \text{ and } \nabla V\cdot \ell =0.$$ 
   This shows that the assumption \textbf{[A$\perp$]} is generic with $f=V/2$, at least in the (large) set $G_{x_0}$. 
   Note that  though the domain $G_{x_0}$ is very large in $B_{x_0}$, without any extra assumption, it might not necessarily be the whole space $B_{x_0}$.  We also mention that  by  \cite[Equation (4.2) and Theorem 3]{day1987r}, $\nabla V(x_0)=0$, the Hessian matrix $\Hess V(x_0)$ at   $x_0$ is positive definite, and all the statements in Theorem \ref{th.sheu}  holds true uniformly in the compacts of $G_{x_0}$ as well as the formula \eqref{eq.R0}. 
   
In conclusion, if  $x_0\in \mathbf R^d$  is   such that $b(x_0)=0$ and \eqref{eq.Ei} holds, then,   for any    smooth bounded  domain  $\Omega$ containing $x_0$ such that   $\overline \Omega \subset   G_{x_0}$ and   {\rm \textbf{[A$_{\mathscr  P_{\rm sp}}$]}} holds, the Eyring-Kramers  formula  \eqref{eq.EK}  holds uniformly for $x$ in the compact subsets of $\mathscr A_{\Omega}(x_0)$.

   \begin{remark}
  The regularity of the quasi-potential is at the heart of the transversal decomposition of $b$ which is   possible in the regions of $B_{x_0}$ where  $V$ is sufficiently smooth (say at least $\mathcal C^1$, see~\cite[Section 4.3]{FrWe}).  More generally, finding the regions where the quasi-potential is smooth is a very challenging task for the success of several asymptotic formulas in statistical physics~\cite{day-darden-85,dayexterior,day1987r,BoRe16}, as it is the case here.  
    \end{remark}

\noindent
\textbf{Acknowledgement}. 
This work was supported by the ANR-19-CE40-0010, Analyse Quantitative de Processus Métastables (QuAMProcs). B.N. is supported by the grant IA20Nectoux from the Projet I-SITE Clermont CAP 20-25.


\bibliographystyle{amsplain}
\bibliography{nonreversible}
\end{document}